\documentclass[11pt,letterpaper,reqno]{amsart}
\usepackage{amsmath,amssymb,amsthm,amsfonts,mathrsfs,url}
\usepackage[mathscr]{eucal}
\usepackage{array}

\usepackage{hyperref}
\hypersetup{
    bookmarks=true,         
    unicode=false,          
    pdftoolbar=true,        
   pdfmenubar=true,        
   pdffitwindow=false,     
   pdfstartview={FitH},    
   pdftitle={Lorentzian polynomials on cones},    
   pdfauthor={Author},     
    pdfcreator={Creator},   
    pdfproducer={Producer}, 
   pdfkeywords={keyword1} {key2} {key3}, 
   pdfnewwindow=true,      
    colorlinks=true,       
    linkcolor=blue,          
    citecolor=blue,        
    filecolor=blue,      
    urlcolor=black    
    }

\theoremstyle{plain}
   \newtheorem{theorem}{Theorem}[section]
   \newtheorem{proposition}[theorem]{Proposition}
   \newtheorem{lemma}[theorem]{Lemma}
   \newtheorem{corollary}[theorem]{Corollary}

   \newtheorem{question}{Question}
\theoremstyle{definition}
   \newtheorem{definition}[theorem]{Definition}
   \newtheorem{example}[theorem]{Example}

\theoremstyle{remark}
   \newtheorem{remark}[theorem]{Remark}

\author[P.~Br\"and\'en]{Petter Br\"and\'en}
\address{Department of Mathematics, KTH Royal Institute of Technology, SE-100 44 Stockholm,
Sweden}
\email{pbranden@kth.se}

\author[J.~Leake]{Jonathan Leake}
\address{Department of Combinatorics and Optimization, University of Waterloo, Canada}
\email{jonathan@jleake.com}


\thanks{PB is a Wallenberg Academy Scholar
  supported by the Knut and Alice Wallenberg
  Foundation, and the G\"oran Gustafsson Foundation.
  JL acknowledges the support of the Natural Sciences and Engineering Research Council of Canada (NSERC), [funding reference number RGPIN-2023-03726]. Cette recherche a \'et\'e partiellement financ\'ee par le Conseil de recherches en sciences naturelles et en g\'enie du Canada (CRSNG), [num\'ero de r\'ef\'erence RGPIN-2023-03726].}

\numberwithin{equation}{section}
\newcommand{\xx}{\mathbf{x}}
\newcommand{\yy}{\mathbf{y}}
\newcommand{\zz}{\mathbf{z}}
\newcommand{\ee}{\mathbf{e}}
\newcommand{\sss}{\mathbf{s}}
\newcommand{\ccc}{\mathbf{c}}
\newcommand{\vv}{\mathbf{v}}
\newcommand{\ttt}{\mathbf{t}}
\newcommand{\ff}{\pol}

\newcommand{\ww}{\mathbf{w}}
\newcommand{\uu}{\mathbf{u}}

\newcommand{\NN}{\mathbb{N}}

\newcommand{\one}{E}
\newcommand{\zero}{K}

\newcommand{\SMOD}{\mathscr{S}}

\newcommand{\PPP}{\mathscr{P}}

\newcommand{\VV}{\mathcal{V}}

\newcommand{\LL}{\mathscr{L}}
\newcommand{\LLL}{\mathrm{L}}

\newcommand{\MM}{\mathrm{M}}

\newcommand{\Hmg}{\mathrm{H}}

\newcommand{\RR}{\mathbb{R}}

\newcommand{\CCC}{\mathscr{K}}
\newcommand{\KKK}{\mathscr{K}}

\DeclareMathOperator{\spn}{span}
\DeclareMathOperator{\bus}{weld}

\newcommand{\dc}{\bar{\partial}}

\def\newop#1{\expandafter\def\csname #1\endcsname{\mathop{\rm
#1}\nolimits}}

\newop{st}
\newop{diag}
\newop{supp}
\newop{rk}
\newop{Deg}
\newop{sub}
\newop{rank}
\newop{Sym}
\newop{sign}
\newop{Int}
\newop{disc}
\newop{mult}
\newop{tr}
\newop{vol}
\newop{Vol}

\newop{pol}
\newop{MAP}
\newop{lk}

\newcommand{\rev}{}

\title[Lorentzian polynomials on cones]{Lorentzian polynomials on cones}

\begin{document}
\begin{abstract} Inspired by the theory of hyperbolic polynomials and Hodge theory, we develop the theory of Lorentzian polynomials on cones. This notion captures the Hodge-Riemann relations of degree zero and one. Motivated by fundamental properties of volume polynomials of Chow rings of simplicial fans, we define a class of multivariate polynomials which we call hereditary polynomials. We give a complete and easily checkable characterization of hereditary Lorentzian polynomials. This characterization is used to give elementary and simple proofs of the Heron-Rota-Welsh conjecture for the characteristic polynomial of a matroid, and the Alexandrov-Fenchel inequalities for convex bodies. 

We then characterize Chow rings of simplicial fans which satisfy the Hodge-Riemann relations of degree zero and one, and we prove that this property only depends on the support of the fan. 

Several different characterizations of Lorentzian polynomials on cones are provided.

\end{abstract}
\maketitle
\thispagestyle{empty}
\setcounter{tocdepth}{2}
\tableofcontents

\section{Introduction}

Over the past \rev{15 years}, methods originating in algebraic geometry have been developed  to solve long-standing conjectures on unimodality and log-concavity in matroid theory \rev{\cite{AHK, ALOV, Lagrangian, top-heavy, BH, Huh-chromatic, HK}}. In the influential paper \cite{AHK}, a Hodge theory of matroids was developed by Adiprasito, Huh and Katz to prove the Heron-Rota-Welsh conjecture on the log-concavity of the characteristic polynomial of a matroid. This paper has generated a surge of research activities at the interface of Hodge theory and combinatorics, and has led to the resolution of  long-standing conjectures in matroid theory, see \cite{Eur-survey,HuhICM}. 

Another successful approach to log-concavity problems originates in the works of Choe, Oxley, Sokal, Wagner, Gurvits and the first author, see \cite{WSur}. This approach uses convexity properties of multivariate stable and hyperbolic polynomials rather than Hodge theory. Recently efforts to combine the two approaches were made by Huh and the first author \cite{BH}, and by Anari, Liu, Oveis Gharan and Vinzant \cite{ALOV-mixing, ALOV}. The theory of Lorentzian polynomials  was developed in \cite{ALOV, BH, Gu}, and used in \cite{ALOV,BH} to prove the strongest of Mason's conjectures on the \rev{(ultra)} log-concavity of the  number of independent sets of a matroid. 

In this paper we study a natural generalization of Lorentzian polynomials to convex cones. The notion is motivated by hyperbolic polynomials \cite{ABG,Garding} and Hodge theory \cite{Baker,Eur-survey,HuhICM}. Indeed, any hyperbolic polynomial is Lorentzian with respect to its hyperbolicity cone, and if a graded vector space $A$ satisfies the so called K\"ahler package, then the volume polynomial associated to $A$ is Lorentzian with respect \rev{to} a convex cone $\CCC$ \rev{(see Example~\ref{ex:C-Lor})}. In fact, the notion of $\CCC$-Lorentzian polynomials captures the Hodge-Riemann relations of degree zero and one. We have made an effort to make this paper self-contained, based on convexity properties of polynomials. \rev{Specifically, many of our results follow from basic facts in classical convexity theory and linear algebra.}
\\[1ex]

\noindent
{\bf  \rev{Outline.}} In Section \ref{Lor-cones} we prove fundamental properties of Lorentzian polynomials on cones, which are used in subsequent sections. \rev{These include sufficient conditions for a polynomial to be $\CCC$-Lorentzian which are suitable for recursive proofs, see Theorem~\ref{engine}.} In Section \ref{hersec} we define a class of multivariate homogeneous polynomials which we call \emph{hereditary polynomials}\rev{, Definition~\ref{def:hered-poly}}. The definition is inspired by fundamental properties of volume polynomials of Chow rings of simplicial fans. We define a natural notion of the Lorentzian property for hereditary polynomials, and give a complete and easily checkable characterization of hereditary  Lorentzian polynomials, Theorem~\ref{mainthm_hereditary}. In Section \ref{matroidsec} we apply this characterization to give a self-contained and elementary proof of the Lorentzian property for the volume polynomial of the Chow ring of a matroid, Theorem~\ref{matroid-hered-Lor}. This is equivalent to the Hodge-Riemann relations of degree zero and one for this Chow ring, which were first proved by Adiprasito, Huh and Katz in \cite{AHK}.
We use this to derive a new proof of the Heron-Rota-Welsh conjecture for the characteristic polynomial of a matroid \rev{in Section~\ref{lc}}, first proved in \cite{AHK}.

In Section \ref{subdsec} we define stellar subdivisions of hereditary polynomials. We show that the stellar subdivision of a polynomial $f$ with respect to a face $S$, is achieved by a canonical bijective linear operator $\sub_S^\ccc$. In Theorem~\ref{support-general} we prove that the hereditary Lorentzian property is invariant under subdivisions. 

In Section \ref{polytsec} we apply the results in Sections \ref{hersec} and \ref{subdsec} to give a new proof of the Lorentzian property of volume polynomials of simple polytopes \cite{Aleksandrov, McMullen}. A consequence is a simple and self-contained proof of the Alexandrov-Fenchel inequalities for convex bodies \cite{Aleksandrov}.  

In Section \ref{chowsec} we give a complete and easily checkable  characterization of Chow rings of simplicial fans which are hereditary Lorentzian, Theorem~\ref{posfanchar}. More generally, we characterize hereditary Lorentzian functionals on Chow rings of simplicial fans, 
Theorem~\ref{genfanchar}.    In Section \ref{suppfans} we prove that the hereditary Lorentzian property only depends on the support of the fan. The corresponding property for Lefschetz fans was recently proved by Ardila, Denham and Huh in \cite{Lagrangian}. 

In  Section \ref{Lorsec} we provide various characterizations of $\CCC$-Lorentzian polynomials \rev{and M-convex sets}, and in Section \ref{top} we prove that the projective space of $\CCC$-Lorentzian polynomials of fixed degree and number variables is compact \rev{and} contractible to a point in its interior. \\[1ex]

\noindent
{\bf  Related work.} The manuscript \cite{BL} by the authors, which is not intended for publication, contains a short and self-contained proof (based on the same methods used in this paper) of the Hodge-Riemann relations of degree zero and one for the Chow ring of a matroid and the Heron-Rota-Welsh conjecture, first proved by Adiprasito, Huh and Katz in \cite{AHK}. We decided to instead include the work in \cite{BL} in the more general setting of this paper. 

While completing this paper, we became aware that Dustin Ross had independently proved the characterization of Chow rings of simplicial fans satisfying the Hodge-Riemann relations of degree zero and one, as well as that this property only depends on the support of the fan.  Our papers are written from two different perspectives, which we believe to be mutually beneficial: The paper of Ross is from the perspective of Lorentzian fans, while our paper is from the perspective of (hereditary) polynomials. Therefore we decided to post our two papers simultaneously.

\section{Lorentzian polynomials on cones}\label{Lor-cones}
Inspired by Hodge theory, matroid theory and the theory of stable polynomials, Lorentzian polynomials were introduced by June Huh and the first author in \cite{BH}. In this section we extend the notion of Lorentzian polynomials to open convex cones. To motivate this definition, we recall the Hodge-Riemann relations for the case of graded rings, see \cite{HuhICM} for a detailed and more general exposition. Suppose 
$$
A= \RR[x_1,\ldots,x_n]/I=\bigoplus_{k=0}^dA^k 
$$
is a graded ring \rev{(with standard grading given by polynomial degree)}, where $A^d$ is one dimensional. Let $\deg : A^d \longrightarrow \RR$ be a linear isomorphism. Suppose further \rev{that} $\CCC$ is an open convex cone\footnote{\rev{In this paper, an \emph{open convex cone} is an open subset $\CCC$ of the ambient vector space for which $\lambda \uu + \mu \vv \in \CCC$ whenever $\uu, \vv \in \CCC$ and $\lambda, \mu$ are positive real numbers.}} in the real vector space $A^1$. For $0 \leq k \leq d/2$, the (mixed) \emph{Hodge-Riemann relations} of degree $k$ say that given any elements $\ell_0,\ell_1, \ldots, \ell_{d-2k}$ in $\CCC$, the bilinear form 
\begin{equation}\label{HRk}
A^k \times A^k \longrightarrow \RR, \ \ \ \ \ (x,y) \longmapsto (-1)^k\deg(xy \ell_1 \ell_2\cdots \ell_{d-2k})
\end{equation}
is positive definite on $\{ x \in A^k : \ell_0\ell_1 \cdots \ell_{d-2k}x =0\}$. 

Notice that $A^0$ is isomorphic to $\RR$ and that $\ell_0 \ell_1 \cdots \ell_d = 0$ in $A$. It follows that for $k=0$, the  Hodge-Riemann relations simply say that $\deg(\ell_1\ell_2 \cdots \ell_{d})>0$ for any $\ell_1, \ldots, \ell_d \in \CCC$. Given this, the Hodge-Riemann relations of degree one say that for any $\ell_1,\ldots, \ell_{d-2} \in \CCC$, the bilinear form 
$$
(x,y) \longmapsto \deg(xy \ell_1\ell_2 \cdots \ell_{d-2})
$$
 has the \emph{Lorentz signature} $(+,-,-,\ldots,-)$, i.e., it is non-singular and has  exactly one positive eigenvalue. This reformulation follows from Cauchy's interlacing theorem.

The Hodge-Riemann relations of degree zero and one may be rephrased in terms of polynomials as follows. The \emph{volume polynomial} of $A$ is the polynomial in $\RR[t_1,\ldots, t_n]$ defined by 
$$
\vol_A(\ttt)= \frac 1 {d!} \deg\left(\left(\sum_{i=1}^nt_ix_i \right)^{\!\! d}\right).
$$
Let $\partial_i$ (or $\partial_{t_i}$) denote the partial derivative with respect to $t_i$, and for $\uu=(u_1, \ldots, u_n) \in \RR^n$ let 
$D_\uu= u_1\partial_1 + \cdots + u_n \partial_n$.
For $\underline{\ell} = (a_1,\ldots, a_n) \in \RR^n$, let $\ell = a_1x_1+\cdots+ a_nx_n\in A^1$. Then 
for any $\ell_1, \ldots, \ell_d \in A^1$, 
$$
\deg(\ell_1 \ell_2\cdots \ell_{d}) = D_{\underline{\ell}_1} D_{\underline{\ell}_2} \cdots D_{\underline{\ell}_d}\vol_A. 
$$

The above discussion motivates the following definition, which encapsulates the Hodge-Riemann relations of degree zero and one. 

\begin{definition}\label{C-def}
Let $\CCC$ be an open convex cone in $\RR^n$. A homogeneous polynomial $f \in \RR[t_1,\ldots, t_n]$ of degree $d \geq 2$ is called $\CCC$-\emph{Lorentzian} if for all 
$\vv_1, \ldots, \vv_{d} \in \CCC$, 
\begin{itemize}
\item[(P)] $D_{\vv_1}\cdots D_{\vv_{d}} f>0$, and 
\item[(HR)] the bilinear form on $\RR^n$,
$$
(\xx,\yy) \longmapsto D_{\xx}D_{\yy}D_{\vv_3}\cdots D_{\vv_d}f
$$
has exactly one positive eigenvalue.
\end{itemize}
\rev{We further define the $\CCC$-Lorentzian polynomials of degree at most $0$ to be the nonnegative constant polynomials, and we define the $\CCC$-Lorentzian polynomials of degree $1$ to be the linear polynomials satisfying (P).}
\end{definition}
By definition, $D_\vv f$ is $\CCC$-Lorentzian whenever $\vv \in \CCC$ and $f$ is $\CCC$-Lorentzian of degree at least one. 

\rev{We will prove in Proposition~\ref{CL-equal} that a polynomial is $\RR_{>0}^n$-Lorentzian if and only if it is Lorentzian as defined in \cite{BH}.}

\begin{example} \label{ex:C-Lor}
    By construction, volume polynomials of graded rings
    that satisfy the Hodge-Riemann relations \rev{(of degrees $0$ and $1$)} are $\CCC$-Lorentzian. This provides an important set of examples from projective geometry, matroid theory and convex geometry, see \cite{HuhICM}. 
    
    Another important class of $\CCC$-Lorentzian polynomials are \emph{hyperbolic polynomials}\rev{,} which originate in PDE-theory \cite{ABG,Garding}. A homogeneous polynomial $f \in \RR[t_1,\ldots, t_n]$ is hyperbolic with respect to a vector $\ee \in \RR^n$ if 
    \begin{itemize}
        \item $f(\ee)>0$, and 
        \item \rev{for each $\xx \in \RR^n$,} all zeros of the univariate polynomial $t \longmapsto f(t\ee-\xx)$ are real. 
    \end{itemize}
For example the determinant $\det(\xx)$, considered as a polynomial on the space of real symmetric matrices, is hyperbolic with respect to the identity matrix.

 The open \emph{hyperbolicity cone} of \rev{a hyperbolic polynomial} $f$ with respect to $\ee$ is the set $\CCC(f,\ee)$ of all $\xx$ for which all zeros of $f(t\ee-\xx)$ are positive. The hyperbolicity cone is convex, and if $\vv \in \CCC(f,\ee)$, then $D_\vv f$ is hyperbolic with respect to $\ee$ and $\CCC(f,\ee) \subseteq \CCC(D_\vv f,\ee)$, see \cite{Garding}. Hence to prove that $f$ is $\CCC(f,\ee)$-Lorentzian, it suffices to prove that quadratic hyperbolic polynomials satisfy (HR). \rev{In fact, a quadratic polynomial is $\CCC$-Lorentzian if and only if it is  hyperbolic with hyperbolicity cone containing $\CCC$, see Lemma~\ref{AF=H} below}.
 
 \rev{We also note that hyperbolic polynomials with hyperbolicity cone containing $\RR_{>0}^n$ are called \emph{real stable} in the literature. Hence $\CCC$-Lorentzian polynomials generalize hyperbolic polynomials with hyperbolicity cone containing $\CCC$ in the same way that Lorentzian polynomials generalize real stable polynomials.}
\end{example}

Notice that (P) is equivalent to \rev{the condition} that all coefficients of the polynomial 
$$
f(s_1\vv_1+s_2\vv_2+\cdots+s_d\vv_d) \in \RR[s_1,\ldots, s_d]
$$
are positive.

\begin{lemma}\label{AF=H}
    Let $\CCC$ be a non-empty open convex cone in $\RR^n$, and let $(\xx,\yy) \longmapsto P(\xx,\yy)$ be a symmetric bilinear form on $\RR^n$ such that $P(\vv,\vv) > 0$ for all $\vv \in \CCC$. Then the following are equivalent:
    \begin{itemize}
        \item[(AF)] $P(\vv,\ww)^2 \geq P(\vv,\vv) P(\ww,\ww)$ for all $\vv,\ww \in \CCC$,
        \item[(AF2)] $P(\vv,\xx)^2 \geq P(\vv,\vv) P(\xx,\xx)$ for all $\vv \in \CCC$ and $\xx \in \RR^n$,
        \item[(Hyp)] \rev{the polynomial $f(\xx) = P(\xx,\xx)$ is hyperbolic with hyperbolicity cone containing $\CCC$,}
        \item[(Lor)] \rev{the polynomial $f(\xx) = P(\xx,\xx)$ is $\CCC$-Lorentzian},
        \item[(H)] $P$ has exactly one positive eigenvalue.
    \end{itemize}
\end{lemma}
\begin{proof}
Suppose $P$ satisfies (AF), and let $\vv \in \CCC$ and $\xx \in \RR^n$. Since $\CCC$ is an open cone, $\ww= \xx+t\vv \in \CCC$ for all $t>0$ sufficiently large. Applying (AF) to $\vv,\ww$ yields (AF2) after a simple calculation. 

\rev{Assume (AF2), and let $\vv \in \CCC$ and $\xx \in \RR^n$. Then
\begin{equation}\label{quadric}
    f(t\vv - \xx) = t^2 \cdot P(\vv,\vv) - 2t \cdot P(\vv,\xx) + P(\xx,\xx),
\end{equation}
which is real-rooted by nonnegativity of the discriminant. Thus $f$ is hyperbolic with hyperbolicity cone containing $\CCC$. Note in fact that (AF2) and (Hyp) are equivalent, since the discriminant condition is equivalent to real-rootedness.}

\rev{Assume (Hyp), and let $\vv \in \CCC$. Since (Hyp) implies (AF2), we have that $P$ is negative semidefinite on the hyperplane $\{\xx : P(\vv,\xx)=0\}$. By Cauchy interlacing and Sylvester's law of inertia, $P$ has at most one positive eigenvalue. Since $P(\vv,\vv) > 0$, $P$ has exactly one positive eigenvalue. Suppose $\xx \in \CCC$. Then $D_{\vv} D_{\xx} f = 2 \cdot P(\vv,\xx) > 0$, since the quadratic \eqref{quadric} has two negative zeros. Hence $f(\xx)$ satisfies (P).}
\rev{The fact that (Lor) implies (H) is immediate from Definition~\ref{C-def}.}


Assume (H), and let $\vv, \ww \in \CCC$. If $\vv$ and $\ww$ are parallel, then (AF) holds. Otherwise extend $\vv,\ww$ to a basis of $\RR^n$. In this basis, the top-left $2 \times 2$ matrix reads 
$$
\begin{pmatrix}
P(\vv,\vv) & P(\vv,\ww) \\
P(\vv,\ww) & P(\ww,\ww)
\end{pmatrix}.
$$
Since $P(\vv,\vv)>0$, Cauchy interlacing and Sylvester's law of inertia implies that this matrix has exactly one positive eigenvalue. The condition (AF) now follows since the determinant of this matrix is nonpositive. 
\end{proof}

\begin{remark}\label{AF-remark}
From Lemma \ref{AF=H} it follows that we may replace (HR) in Definition~\ref{C-def} by the following equivalent condition. Let $\overline{\CCC}$ be the closure of $\CCC$ in the Euclidean topology on $\RR^n$. 
\begin{itemize}
\item[(AF)] For all $\vv_1, \ldots, \vv_{d} \in \overline{\CCC}$,
$$(D_{\vv_1}D_{\vv_2} D_{\vv_3}\cdots D_{\vv_{d}} f)^2 \geq (D_{\vv_1}D_{\vv_1} D_{\vv_3}\cdots D_{\vv_{d}} f)\cdot (D_{\vv_2}D_{\vv_2} D_{\vv_3}\cdots D_{\vv_{d}} f).$$
\end{itemize}
In particular, if $f$ is $\CCC$-Lorentzian and $\uu, \vv \in \overline{\CCC}$
, then the sequence $a_k=D_\uu^kD_\vv^{d-k} f$, $0\leq k \leq d$, is \emph{log-concave}, i.e., 
\begin{equation}\label{af-lc}
a_k^2 \geq a_{k-1}a_{k+1}, \ \ \ \ \ 0 \leq k \leq d. 
\end{equation}
It is no coincidence that (AF) is reminiscent of the \emph{Alexandrov-Fenchel inequalities} for convex bodies \cite{Aleksandrov}. 
\end{remark}

The following remark proves that the space of $\CCC$-Lorentzian polynomials is closed in the linear space of all $d$-homogeneous polynomials in $\RR[t_1,\ldots, t_n]$. 
\begin{remark}\label{closedc}
We may replace (P) with the seemingly weaker condition 
\begin{itemize}
\item[(N)] $D_{\vv_1}\cdots D_{\vv_{d}} f \geq 0$, for all $\vv_1, \ldots, \vv_{d} \in \overline{\CCC}$. 
\end{itemize}
Indeed, assume (N) and that $D_{\vv_1}\cdots D_{\vv_{d}} f = 0$
for some $\vv_1, \ldots, \vv_{d} \in \CCC$. Let $\uu_1 \in \CCC$. Since $\CCC$ is open, there is $\epsilon >0$ such that
$$
D_{\vv_1+t\uu_1} \rev{D_{\vv_2}} \cdots D_{\vv_{d}}f = tD_{\uu_1} \rev{D_{\vv_2}} \cdots D_{\vv_{d}}f\geq 0, 
$$
for all  $|t| <\epsilon$. Hence 
$
D_{\uu_1}D_{\vv_2}\cdots D_{\vv_{d}}f =0 
$
for all $\uu_1 \in \CCC$. By applying the same argument to $\vv_2,\ldots, \vv_d$, 
$
D_{\uu_1}D_{\uu_2}\cdots D_{\uu_{d}}f =0
$
for all $\uu_1,\ldots, \uu_d \in \CCC$. Hence $f \equiv 0$, since $\CCC$ is open.
\end{remark}

\rev{The following is a generalization of \cite[Thm. 2.10]{BH} and \cite[Lem. 2.2]{ALOV} from Lorentzian to $\CCC$-Lorentzian polynomials.}

\begin{proposition}\label{comp}
Let $A : \RR^m \to \RR^n$ be a linear map. If $f \in \RR[t_1,\ldots,t_n]$ is $\CCC$-Lorentzian, then the polynomial $g$ defined by $g(\xx)= f(A\xx)$ is $A^{-1}(\CCC)$-Lorentzian.  
\end{proposition}
\begin{proof}
    \rev{
    Given a cone $\CCC$, Remark~\ref{AF-remark} implies that a polynomial is $\CCC$-Lorentzian if and only if it satisfies (P) and (AF) with respect to $\CCC$. Notice that 
    \begin{equation}\label{us}
    D_{\vv_1}D_{\vv_2}\cdots D_{\vv_d} g = D_{A\vv_1}D_{A\vv_2}\cdots D_{A\vv_d}f.
    \end{equation}
    Thus (P) and (AF) for $g$ follow directly from (P) and (AF) for $f$.
    }
\end{proof}

Recall that the \emph{lineality space} of an open convex  cone $\CCC$ in $\RR^n$ is $L_\CCC=\overline{\CCC}\cap -\overline{\CCC}$, i.e., the largest linear space contained in the closure of $\CCC$. Given a homogeneous polynomial $f \in \RR[t_1,\ldots,t_n]$, we define the \emph{lineality space of $f$}\rev{, denoted $L_f$,} to be the set of all $\vv \in \RR^n$ for which $D_{\vv} f \equiv 0$, i.e., the set of all $\vv$ for which 
$$
f(\ttt+\vv) = f(\ttt), \mbox{ for all } \ttt \in \RR^n.
$$
\begin{proposition}\label{lineal}
Let $\CCC$ be an open convex cone in $\RR^n$. If $f$ satisfies (P), then $L_{\CCC} \subseteq L_f$.
\end{proposition}

\begin{proof}
If $\vv \in \CCC$ and $\ww \in L_{\CCC}$, then 
$$
D_\ww f(\vv)= \frac{D_\ww D_\vv^{d-1}f}{(d-1)!} \geq 0 \ \ \mbox{ and } \ \ -D_{\ww} f(\vv)= \frac {D_{-\ww}D_\vv^{d-1}f}{(d-1)!} \geq 0.
$$
 Hence $D_\ww f(\vv) =0 $ for all $\vv \in \CCC$. \rev{Since $D_\ww f$ is a polynomial which is $0$ on the open set $\CCC$, we have that $D_\ww f \equiv 0$ as desired.}
\end{proof}

A matrix $A=(a_{ij})_{i,j=1}^n$ with nonnegative off-diagonal entries  is called \emph{irreducible} if for all distinct $i,j$ there is a sequence $i=i_0,i_1,i_2,\ldots, i_\ell=j$  such that $i_{k-1}\neq i_{k}$ for all $1 \leq k \leq \ell$, and 
$
a_{i_0 i_1} a_{i_1 i_2} \cdots a_{i_{\ell-1}i_{\ell}}>0.
$
By translating such a matrix by a positive multiple of the identity matrix, the Perron-Frobenius theory \cite[Chapter~1]{BP} guarantees that $A$ has a unique eigenvector (up to multiplication by positive scalars) whose entries are all positive. Moreover the corresponding eigenvalue is simple and is the largest eigenvalue of $A$. 

If $A$ and $B$ are symmetric matrices of the same size, we write $A \preceq B$ if $B-A$ is positive semidefinite.  Recall that the \emph{Hessian} of $f$ at $\xx$ is the matrix $\nabla^2 f (\xx) = (\partial_i\partial_j f(\xx))_{i,j=1}^n$. 

The following lemma is,  in essence, taken from \cite[Prop. 3]{AFI}. 
\begin{lemma}\label{indlemma}
Let $f\in \RR[t_1,\ldots, t_n]$ be a homogeneous polynomial of degree $d \geq 3$, and let $\xx \in \RR_{>0}^n$. If
\begin{enumerate}
\item  $\partial_i f(\xx)>0$ for all $i$, and 
\item the Hessian  of $\partial_i f$ at $\xx$ has exactly one positive eigenvalue for all $i$, and 
\item the Hessian of $f$ at $\xx$ \rev{has nonnegative off-diagonal entries} and is irreducible, 
\end{enumerate}
then the Hessian of $f$ at $\xx$ has exactly one positive eigenvalue. 
\end{lemma}
\begin{proof}
 If $g$ is a $d$-homogeneous polynomial and $g(\xx)>0$, then the following three statements are equivalent 
\begin{itemize}
\item[(a)] the Hessian of $g$ at $\xx$ has exactly one positive eigenvalue,
\item[(b)]   the Hessian of $g^{1/d}$  is negative semidefinite at $\xx$, 
\item[(c)] the matrix $d\cdot g \cdot \nabla^2g - (d-1)\cdot \nabla g (\nabla g)^T$ is negative semidefinite at $\xx$,
\end{itemize}
see e.g. \cite[Prop. 2.33]{BH}. 

Suppose $\xx$ and $f$ are as in the hypotheses of the lemma. 
Then, by (c), 
$$
(d-1)\cdot \partial_i f \cdot \nabla^2 \partial_i f  \preceq (d-2) \cdot \nabla \partial_i f (\nabla \partial_i f)^T.
$$
Euler's identity, 
$
d \cdot f(\xx) = \sum_{i=1}^n x_i \cdot \partial_i f(\xx),
$
 yields 
$$
(d-2)\cdot \nabla^2 f = \sum_{i=1}^n x_i \nabla^2  \partial_i f  \preceq  \sum_{i=1}^n \frac {x_i}{\partial_i f} \frac {(d-2)}{(d-1)} \nabla \partial_i f (\nabla \partial_i f)^T.
$$
Rewrite the above inequality as
$
(d-1) \cdot \nabla^2 f \preceq (\nabla^2 f)  \Lambda (\nabla^2 f), 
$ 
where $\Lambda$ is the diagonal matrix $\diag({x_1}/{\partial_1 f}, \ldots, {x_n}/{\partial_n f})$. For the matrix $B= \Lambda^{1/2} (\nabla^2 f) \Lambda^{1/2}$,  this implies 
$
B^2 -(d-1)B \succeq 0.
$
Hence no eigenvalue of $B$ lies in the open interval $(0,d-1)$. The matrix $B$ is irreducible and has nonnegative off-diagonal entries, so the Perron-Frobenius theorem applies to $B$. 
Notice that $\Lambda^{-1/2} \xx$ is a positive eigenvector of $B$, and the corresponding eigenvalue is $d-1$. Hence   $d-1$ is the unique largest eigenvalue of $B$ afforded by  the Perron-Frobenius theorem. We conclude that $B$, and thus also $\nabla^2 f(\xx)$, has exactly one positive eigenvalue.
\end{proof}
We say that an open convex cone $\CCC$ in $\RR^n$ is \emph{effective} if $\CCC= \CCC \cap \RR_{>0}^n + L_\CCC$. 

\begin{example}\rev{
Any open convex cone contained in the positive orthant is trivially effective. Let $\CCC$ be an open half-plane in $\RR^2$ whose boundary contains the origin. Then $\CCC$ is effective if and only if $\CCC \cap \RR_{>0}^2 \neq 0$. }
    \end{example}
The following theorem provides sufficient conditions on a polynomial to be Lorentzian with respect to an effective cone. 

\begin{theorem}\label{engine}
Let $f\in \RR[t_1,\ldots, t_n]$ be a homogeneous polynomial of degree $d \geq 3$, and let $\CCC$ be an open, convex and effective cone in $\RR^n$. If 
\begin{enumerate}
\item $D_{\vv_1}\cdots D_{\vv_d}f >0$ for all $\vv_1, \ldots, \vv_d \in \CCC$, and 
\item  the Hessian of $D_{\vv_1} \cdots D_{\vv_{d-2}}f$ is irreducible and its off-diagonal entries are nonnegative for all  $\vv_1,\ldots,  \vv_{d-2} \in \CCC$,  and 
\item $\partial_i f$ is $\CCC$-Lorentzian for all $i$, 
\end{enumerate}
then $f$ is $\CCC$-Lorentzian.
\end{theorem}

\begin{proof}
Let $\vv_1, \ldots, \vv_{d-3} \in \CCC$, and consider the cubic  $g= D_{\vv_1}\cdots D_{\vv_{d-3}} f$. 
Since $\partial_i f$ is $\CCC$-Lorentzian, it follows from Definition~\ref{C-def} that so is $\partial_i g$. By choosing $\vv_{d-2}=\xx \in \CCC$, it follows from (2) that 
the Hessian $\nabla^2g(\xx)$ is irreducible and its off-diagonal entries are nonnegative. 

By (1) and \rev{Proposition} \ref{lineal},  $g(\xx+\ww)=g(\xx)$ for all $\xx \in \RR^n$ and $\ww \in L_\CCC$.  Since $\CCC$ is effective, we may assume  $\xx \in \CCC \cap \RR_{>0}^n$. Lemma \ref{indlemma} then implies that the Hessian of $g$ at $\xx$ has exactly one positive eigenvalue. 

The \rev{theorem} now follows since the Hessian of $D_{\vv_1}\cdots D_{\vv_{d-3}} f$ at $\xx$ is equal to the Hessian of $D_{\vv_1}\cdots D_{\vv_{d-3}}D_\xx f$. 
\end{proof}

\section{Hereditary polynomials}\label{hersec}

In this section we introduce \emph{hereditary polynomials}\footnote{\rev{The definition of hereditary polynomial has slightly changed since a previous version of this paper. See Definition~\ref{def:hered-poly}.}}.  These capture fundamental properties of volume polynomials of Chow rings of simplicial fans. We characterize  hereditary Lorentzian polynomials, which in Section \ref{chowsec}  will lead  to a characterization of Chow rings that satisfy the Hodge-Riemann relations of degree zero and one.

Let $\Delta$ be an abstract simplicial complex on a finite set $V$, and let  $L$ be a linear subspace \rev{of} $\RR^V$. 
For $S \in \Delta$, let $V_S$  
be the set of vertices of the \emph{link}, $$\lk_\Delta(S)=\{T \rev{ \in \Delta} : T\cap S=\varnothing \mbox{ and } S \cup T \in \Delta\},$$ of $S$ in $\Delta$. The \emph{skeleton}, $\tau\Delta$, of $\Delta$ is the simplical complex obtained by removing all facets from $\Delta$. 
If 
\begin{equation}\label{simpl}
\{ (\ell_i)_{i \in T} : \ell \in L\} = \RR^T, \ \ \mbox{ for all } T \in \Delta,
\end{equation}
we say the pair $(\Delta,L)$ is \emph{hereditary}. \rev{That is, $(\Delta, L)$ is hereditary if for every $T \in \Delta$, the coordinate projection of $L$ onto $\RR^T \subseteq \RR^V$ (defined by deleting coordinates indexed by $V \setminus T$) is surjective. Further, we say the pair $(\Delta, L)$ is \emph{weakly hereditary} if $(\tau\Delta, L)$ is hereditary.}

Given a homogeneous polynomial $f \in \RR[t_i : i \in V]$ of degree $d\geq 1$, we define a simplicial complex 
$$
\Delta_f= \{ S \subseteq V :  \partial^Sf \not \equiv 0\}.
$$

\begin{definition} \label{def:hered-poly}
    Let $\Delta$ be a simplicial complex on $V$, let $L$ be a linear subspace of $\RR^V$. \rev{We define $\PPP^k(\Delta,L)$ to be the linear subspace of $\RR[t_i : i \in V]$ consisting of all polynomials $f$ of degree $k$ for which $L \subseteq L_f$ and $\Delta_f \subseteq \Delta$. If $(\Delta,L)$ is (weakly) hereditary then the polynomials in $\PPP^k(\Delta,L)$ are called \emph{(weakly) hereditary}. }
    
    \rev{Equivalently,  a polynomial $f$ is \emph{(weakly) hereditary} if $(\Delta_f,L_f)$ is (weakly) hereditary.}
\end{definition}

\begin{example}\rev{
Let $f=(v_1t_1+\cdots+v_nt_n)^n$, where each entry of $\vv=(v_1,\ldots,v_d)$ is a nonzero real number. Then $\Delta_f$ is a simplex and $L_f$ is the orthogonal complement of $\RR\vv$. It follows that $f$ is weakly hereditary since for each facet $F=\{1, \ldots, n\} \setminus \{i \}$ of $\tau\Delta_f$, and given real numbers $x_j$, $j \in F$, the vector $(x_1, \ldots, x_n)$, where 
$$x_i = -\frac 1 {v_i} \sum_{j \neq i} v_jx_j 
$$
lies in $L_f$. Notice that $f$ is weakly hereditary but not hereditary.}
\end{example}

\begin{lemma}\label{herbydiff}
\rev{
Suppose $g \in \RR[\partial_1, \ldots, \partial_n]$ is homogeneous of degree $k \geq 1$, and that 
$f \in \RR[t_1, \ldots, t_n]$ is a weakly hereditary polynomial. Then the polynomial $h(\ttt)=g(\partial)f(\ttt)$ is hereditary. }
\end{lemma}
\begin{proof}
 Clearly $L_f \subseteq L_h$ and $\Delta_h \subseteq \tau \Delta_f$. Hence $h$ is hereditary.  
\end{proof}
\rev{
If $f \in \RR[t_1,\ldots, t_n]$ and $S \subseteq \{1,\ldots,n\}$, define 
$$
f^S = \partial^Sf\big|_{t_i=0, i\in S}.
$$
Note that $\Delta_{f^S} = \lk_{\Delta_f}(S)$ and $L_{f^S} \supseteq L_S$ where we define
\[
    L_S = \{(\ell_i)_{i \in V_S} : \ell \in L \text{ and } \ell_i = 0, ~ \forall i \in S\}.
\]
\begin{lemma}\label{hereditary}
If $f$ is (weakly) hereditary and $S \in \Delta_f$, then 
$
f^S
$
is (weakly) hereditary.
\end{lemma}
}

\rev{
Given weakly hereditary $(\Delta, L)$, fix $T \in \tau\Delta$ and $S \in \lk_{\tau\Delta}(T)$. For all $i \in S$, let $\ell^{(i)} \in L_T \subseteq \RR^{V_T}$ be such that $\ell_i^{(i)} = 1$ and $\ell_j^{(i)} = 0$ for all $j \in S \setminus \{i\}$. We define a linear map $\pi_S^T: \RR^{V_T} \to \RR^{V_{T \cup S}}$ by
\begin{equation}\label{piSdef}
    \pi_S^T\big((x_i)_{i \in V_T}\big) = (y_j)_{i \in V_{T \cup S}}, \quad \text{where} \quad (y_j)_{j \in V_T} = (x_j)_{j \in V_T} - \sum_{i \in S} x_i \ell^{(i)}.
\end{equation}
That is, $\yy \in \RR^{V_T}$ is defined via $\yy = \xx - \sum_{i \in S} x_i \ell^{(i)}$, and then $\pi_S^T(\xx)$ is defined to be the coordinate projection of $\yy$ onto $\RR^{V_{T \cup S}}$. Hence for $f \in \PPP^d(\Delta,L)$ we have
\begin{equation}\label{derproj}
\partial^Sf^T(\ttt)= f^{T \cup S}(\pi^T_S(\ttt)), \quad \text{for each } S \in \tau \Delta_{f^T}
\end{equation}
since $\ell^{(i)} \in L_{\partial^S f^T}$ for all $i \in S$. We will often use the simpler notation $\pi_S = \pi^\varnothing_S$.}

\begin{lemma}
If $f$ is a weakly hereditary polynomial of degree $d$, then $\Delta_f$ is pure of dimension $d-1$.
\end{lemma}

\begin{proof}
    \rev{
 If $S \in \Delta_f$ has fewer than $d$ elements, then $\partial^Sf \not \equiv 0$ so that 
$$
0 \not \equiv \partial^Sf(\ttt) = f^S(\pi_S(\ttt)),
$$
by \eqref{derproj}. But then $f^S \not \equiv 0$ is a weakly hereditary polynomial in a set of variables $t_j$ such that  $j \notin S$. If $t_j$ is such a variable, then 
$$
0 \not \equiv \partial_j f^T = \partial^{S \cup \{j\}}f\big|_{t_i=0, i\in S}.
$$
Hence $S$ is not maximal in $\Delta_f$. }
\end{proof}

There is a canonically defined open convex cone associated to any weakly hereditary polynomial. 
\begin{definition}\label{hericone}
Let $f \in \RR[t_i: i \in V]$ be a weakly hereditary polynomial. Define an open (possibly empty) convex cone $\CCC_f$ in $\RR^V$ recursively as follows.
\begin{itemize}
\item[(1)] If $f$ is linear, then $\CCC_f = \{ \vv \in \RR^V : f(\vv)>0\}$.
\item[(2)] If $\deg(f)>1$, then 
$$
\CCC_f= (\RR_{>0}^V +L_f)\cap \{ \vv \in \RR^V : \pi_{\{i\}}(\vv) \in \CCC_{f^{\{i\}}} \mbox{ for all } i \in V\}. 
$$
\end{itemize}
\end{definition}
It is straightforward to see that Definition \ref{hericone} does not depend on the choices of the projections $\pi_{\{i\}}$. By definition, $\CCC_f$ is effective whenever $d >1$. Also, by construction,
\rev{
\begin{equation}\label{incl}
\pi^T_S(\CCC_{f^T}) \subseteq \CCC_{f^{T \cup S}} \qquad \mbox{ and } \qquad \pi^T_S(L_{f^T}) \subseteq L_{f^{T \cup S}}, 
\end{equation}
for all $T \in \Delta_f$ and $S \in \tau \Delta_{f^T}$.}
\begin{definition}
We say that a weakly hereditary polynomial $f$ is \emph{positive} if $f^F>0$ for all facets $F$ of $\Delta_f$. \rev{(Notice that $f^F$ is a constant polynomial whenever $F$ is a facet of $\Delta_f$.)}
\end{definition}
Notice that if $f$ is positive, then we may replace (1) \rev{Here ref X wants to replace (1) by (2), but I don't agree} in Definition \ref{hericone} with $\CCC_f= \RR_{>0}^V +L_f$ \rev{for linear $f$}.

\begin{definition}\label{helor}
A weakly hereditary polynomial $f$ is \emph{hereditary Lorentzian} if 
for each $S \in \tau\Delta_f$, the polynomial 
$f^S$ is $\CCC_{f^S}$-Lorentzian. 
\end{definition}

\begin{definition}
    We say that a pure $(d-1)$-dimensional simplicial complex $\Delta$ is \emph{H-connected} \rev{(or \emph{hereditary connected})} if $\lk_\Delta(S)$ is connected for each $S \in \Delta$ of size at most $d-2$. Any zero-dimensional simplicial complex is considered to be \rev{H}-connected. 
\end{definition}

The next theorem is the main theorem of this section\rev{, and it can be thought of as a generalization of \cite[Thm. 2.25]{BH} or \cite[Thm. 3.2]{ALOV}}. It characterizes weakly hereditary Lorentzian polynomials into two conditions that are simple to check.

\begin{theorem} \label{mainthm_hereditary}
Suppose $f \in \RR[t_i : i \in V]$ is a weakly hereditary polynomial of degree $d \geq 2$, and that $\CCC_f \neq \varnothing$. Then $f$ is hereditary Lorentzian if and only if 
\begin{itemize}
\item[(C)] $\tau\Delta_f$ is \rev{H}-connected,\footnote{\rev{The reason we only need this for $\tau\Delta_f$ instead of $\Delta_f$ is because condition (Q) already gives us all we need in the case that $d=2$ (condition (C) is vacuous in this case). Condition (C) is then used in the induction for $d \geq 3$, and only connectivity for $\tau \Delta_f$ is needed. (See also Remark~\ref{main-remark}.)}} and 
\item[(Q)] for each $S \in \Delta_f$ with $|S|= d-2$, the Hessian of \rev{the} polynomial $f^S$ has at most one positive eigenvalue.

\end{itemize}

\end{theorem}
\begin{proof}
We start by proving sufficiency of (C) and (Q) by induction over $d \geq 2$. 
Suppose $d=2$, and let $i \in V$. Then 
$
\partial_if(\vv)=f^{\{i\}}(\pi_{\{i\}}(\vv))>0
$
for all  $\vv \in \CCC_f$,  since $\pi_{\{i\}}(\vv) \in \CCC_{f^{\{i\}}}$. Hence $f^{\{i\}}$ is $\CCC_{f^{\{i\}}}$-Lorentzian.  Also 
$$
 f (\vv) = \frac 1 2\sum_{i\in V} v_i \cdot \partial_if(\vv) >0,
$$
since $\CCC_f$ is effective. By Lemma \ref{AF=H} and (Q), $f$ is $\CCC_f$-Lorentzian. 

Suppose $d \geq 3$. Then $f^S$ satisfies (C) and (Q) for all nonempty $S \in \tau\Delta_f$. By induction,
$f^S$ is $\CCC_{f^S}$-Lorentzian for all nonempty $S \in \tau\Delta_f$. It remains to prove that $f$ is $\CCC_f$-Lorentzian. To do this  we  apply Theorem \ref{engine}.   
To verify (1) of Theorem~\ref{engine}, we prove that for all $\vv_1, \ldots, \vv_d$ in $\CCC_f$ the polynomial 
$
f(s_1\vv_1+ \cdots+s_d\vv_d)
$
has positive coefficients. Since $\CCC_f$ is effective we may assume $\vv_1, \ldots, \vv_d \in \CCC_f \cap \RR_{>0}^n$. Let $\yy=s_1\vv_1+ \cdots+s_d\vv_d$. By Euler's formula, 
$$
d \cdot f(\yy)= \sum_{i \in V} y_i \cdot f^{\{i\}} \big( s_1\pi_{\{i\}}(\vv_1)+\cdots+ s_d\pi_{\{i\}}(\vv_d)\big).
$$
By assumption all coefficients of $y_i$ are positive for each $i \in V$, and by induction all coefficients of 
$
f^{\{i\}} \big( s_1\pi_{\{i\}}(\vv_1)+\cdots+ s_d\pi_{\{i\}}(\vv_d)\big)
$
are positive since $\pi_{\{i\}}(\vv_j) \in \CCC_{f^{\{i\}}}$ for all $j$. This proves the claim. 

Next we claim that for each $\{i,j \} \in \tau\Delta_f$ and all  $\vv_1, \ldots, \vv_{d-2}$ in $\CCC_f$, the polynomial 
$$
(\partial_i\partial_j f)(s_1\vv_1+ \cdots+s_{d-2}\vv_{d-2})=f^{\{i,j\}} \big(s_1\pi_{\{i,j\}}(\vv_1)+\cdots+ s_{d-2}\pi_{\{i,j\}}(\vv_{d-2})\big)
$$
has positive coefficients. The claim will prove the nonnegativity of the off-diagonal entries of the Hessian in Theorem \ref{engine} (2). The claim follows by induction as above, since $\pi_{\{i,j\}}(\vv_k) \in \CCC_{f^{\{i,j\}}}$ for each $k$. By H-connectivity of $\tau \Delta_f$, the Hessian of 
$D_{\vv_1}\cdots D_{\vv_{d-2}}f$ is irreducible. Condition (3) of Theorem \ref{engine} follows by induction since 
$
\partial_i f(\ttt) = f^{\{i\}} ( \pi_{\{i\}}(\ttt)) 
$
is $\CCC_f$-Lorentzian for each $i \in V$ by Proposition~\ref{comp} and \eqref{incl}. Hence sufficiency  follows from Theorem \ref{engine}. 

For the converse, we need to prove that for a weakly hereditary Lorentzian polynomial $f$, $\tau\Delta_f$ is H-connected. For $d=2$, (C) is obvious, since $\tau\Delta_f$ is zero-dimensional. Suppose $d\geq 3$. It suffices to prove that 
$
\tau\Delta_f
$
is connected. Let $\vv \in \CCC_f$. Then  $g=D_{\vv}^{d-2}f$ is a $\CCC_f$-Lorentzian polynomial. \rev{Moreover if $\{i,j\} \in  \tau\Delta_f$, then 
$$
\partial_i \partial_j g= D_\vv^{d-2} \partial_i \partial_j f = (d-2)!f^{\{i,j\}}(\pi_{\{i,j\}}(\vv))>0, 
$$
since $f^{\{i,j\}}$ is $\pi_{\{i,j\}}(\CCC_f)$-Lorentzian.} 
 If $\tau\Delta_f$ is not connected, then we may write the Hessian of $g$ as 
 a block matrix $A\oplus B$ with nonnegative off-diagonal entries, where $A$ has exactly one positive eigenvalue and $B=(b_{ij})$ is a negative semidefinite irreducible matrix for which $b_{ij} >0$ for some $\{i,j\} \in  \tau\Delta_f$. By Perron-Frobenius the largest eigenvalue of $B$ is simple. Hence the kernel of $B$ has dimension at most one. Since $\{i,j\} \in \tau\Delta_f$ and $f$ is weakly hereditary,  for each $(x, y) \in \RR^2$ there exists $\ell \in L_f \subseteq L_g$ such that $\ell_i=x$ and $\ell_j=y$. Hence the kernel of $B$ is at least two-dimensional. Hence $\tau\Delta_f$ is connected. 
\end{proof}

\begin{remark}
    If $f$ is positive, then we may replace (C) in Theorem~\ref{mainthm_hereditary} with the condition that $\Delta_f$ is H-connected.
\end{remark}

\begin{remark}\label{main-remark}
We may reformulate Theorem \ref{mainthm_hereditary} as follows. Let $f$ be a weakly hereditary polynomial of degree $d \geq 2$ and suppose $\CCC_f$ is nonempty. Then $f$ is hereditary Lorentzian if and only if 
\begin{itemize}
    \item $d=2$ and the Hessian of $f$ has at most one positive eigenvalue, or
    \item $d\geq 3$ and $\Delta_f$ is connected and $f^{\{i\}}$ is hereditary Lorentzian for each $i \in V$.
\end{itemize}
\end{remark}

\begin{proposition}\label{main-conv}
Suppose $f \in \RR[t_i : i \in V]$ is a weakly hereditary polynomial of degree $d \geq 2$. If $\CCC$ is a nonempty open convex cone such that  $\pi_S(\CCC)$ is effective and $f^S$ is $\pi_S(\CCC)$-Lorentzian  for all $S \in \tau \Delta_f$, then $\CCC \subseteq \CCC_f$ and $f$ is hereditary Lorentzian.  
\end{proposition}
\begin{proof}
If $\CCC$ satisfies the hypothesis in the statement, then so does $\CCC+L_f$. Hence we may assume $L_f \subseteq \overline{\CCC}$. Since $f$ is weakly hereditary and $\CCC$ is open and nonempty, $\pi_S(\CCC)$ is nonempty for all $S \in \tau \Delta_f$. By induction we deduce \rev{$\pi_S(\CCC) \subseteq \CCC_{f^S}$ for all nonempty $S \in \tau \Delta_f$. If $\vv \in \CCC$, then $\vv \in \RR_{>0}^V+ L_f$, since $\CCC$ is effective. Moreover $\pi_{\{i\}}(\vv) \in \pi_{\{i\}}(\CCC) \subseteq \CCC_{f^{\{i\}}}$. Hence $\vv \in \CCC_f$ by the definition of $\CCC_f$.  As in the proof of Theorem \ref{mainthm_hereditary}, (C) follows from that $f^S$ is $\pi_S(\CCC)$-Lorentzian. Also (Q) holds since the quadratics are $\pi_S(\CCC)$-Lorentzian. Hence the theorem follows from Theorem \ref{mainthm_hereditary}. }
\end{proof}

By Euler's formula we have the following recursive formula for weakly hereditary polynomials.
\begin{equation}\label{recdef}
(d-|S|)\cdot f^S(\ttt)= \sum_{i \in V_S}t_i \cdot f^{S \cup\{i\}}(\pi_{\{i\}}^S(\ttt)), \ \ \ S \in \Delta_f. 
\end{equation}
Hence a weakly hereditary polynomial is uniquely determined by the numbers 
$w(S)\rev{:=}\partial^S f$, $|S|=d$. For each facet $F$ of $\tau\Delta_f$, 
$$
f^F(\ttt) = \sum_{i \not \in F} w(F\cup\{i\})t_i, 
$$
and by construction this linear form is identically zero on $\pi_F(L_f)$. Theorem~\ref{uniquepoly} below shows that this exactly determines when the numbers $w(S)$ determine a (unique) weakly hereditary polynomial.

\begin{lemma}\label{Euler-cons}
Suppose $f \in \RR[t_1, \ldots, t_n]$ is a homogeneous polynomial of degree $d$, and that 
$
df = \sum_{i=1}^n t_i Q_i,
$
where $Q_1, \ldots, Q_n$ are homogeneous polynomials of degree $d-1$ for which 
$
\partial_i Q_j = \partial_j Q_i$, for all $i,j$. 
Then $Q_i= \partial _i f$, for all $i$. 
\end{lemma}
\begin{proof}
By Euler's identity,
\begin{align*}
d \partial_j f &= \partial_j(t_jQ_j) -t_j\partial_jQ_j+ \sum_{i=1}^n t_i\partial_jQ_i  \\
&= Q_j + \sum_{i=1}^n t_i\partial_iQ_j = Q_j+(d-1)Q_j= dQ_j.
\end{align*}
\end{proof}

The next theorem is analogous to \emph{Minkowski weights} and the \emph{balancing condition} for Chow rings of simplicial fans, see \cite[Section 5]{AHK}.

\begin{theorem} \label{uniquepoly}
Let $\Delta$ be a pure $(d-1)$-dimensional simplicial complex on $V$, let $L$ be a subspace of $\RR^V$ such that $(\Delta,L)$ is weakly hereditary, and let $w$ be a real-valued function on the $d$-element subsets of $V$ which is non-zero precisely on the facets of $\Delta$. Then there is at most one weakly hereditary polynomial $f$ for which
$$
\partial^Sf = w(S), \ \ \ \ \mbox{ for all } |S|=d,
$$
$\Delta_f = \Delta$, and $L \subseteq L_f$. 

Moreover, this polynomial exists if and only if for each facet $F$ of $\tau\Delta$, the linear form
$$
\sum_{i \not \in F} w(F\cup\{i\})t_i
$$
is identically zero on $L_F$. If further $(\Delta,L)$ is hereditary, then $f$ is hereditary.
\end{theorem}

\begin{proof} 
That there is at most one weakly hereditary polynomial with these properties follows from the above discussion, since a weakly hereditary polynomial $f$ is determined by $\partial^S f$ for all $S \subset V$ for which $|S| = d$. Further, if such a weakly hereditary polynomial $f$ exists, then for each facet $F$ of $\tau\Delta$,
$$
f^F(\ttt) = \sum_{i \not \in F} w(F\cup\{i\})t_i
$$
is identically zero on $\pi_F(L_f) \supseteq \pi_F(L) = L_F$.

It remains to prove that the condition on linear forms is enough to guarantee the existence of the weakly hereditary polynomial $f$.
To this end, we define a family of polynomials $g^S \in \RR[t_i : i \in V_S]$ recursively by 
\begin{equation}\label{gdef}
(d-|S|)\cdot g^S(\ttt)= \sum_{i \in V_S}t_i \cdot g^{S\cup\{i\}}(\pi_{\{i\}}^S(\ttt)), \ \ \ S \in \Delta,
\end{equation}
where 
$
g^F=w(F)  
$
for each facet $F$ of $\Delta$. 

 We prove by induction over $d-|S|$ that 
$g^{S}$ is weakly hereditary and that $L_S \subseteq L_{g^{S}}$. The case when $|S|=d-1$ is clear since then $L_S \rev{\subseteq} L_{g^{S}}$ by hypothesis. Suppose $S \in \Delta$ with $|S|< d-1$, and that $g^T$ is weakly hereditary with $L_{g^T}\supseteq L_T$ for all $T \in \Delta$ of size greater than $|S|$. We claim $\partial_i g^S(\ttt)= g^{S \cup\{i\}}(\pi^S_{\{i\}}(\ttt))$ for each $i \in V_S$. By Lemma \ref{Euler-cons} and Euler's formula this holds if and only if 
$$
\partial_j \left(g^{S \cup\{i\}}(\pi_{\{i\}}^S(\ttt))\right) = \partial_i \left(g^{S \cup\{j\}}(\pi_{\{j\}}^S(\ttt))\right)
$$
for each $\{i,j\} \in \Delta_S$. For $|S|=d-2$ this is true by hypothesis. For $|S|<d-2$, we get by induction 
$$
\partial_j \left(g^{S \cup\{i\}}(\pi_{\{i\}}^S(\ttt))\right)=g^{S \cup\{i,j\}}(\pi_{\{j\}}^{S\cup\{i\}}\pi_{\{i\}}^S(\ttt)).
$$
By the induction hypothesis it follows that $g^{S \cup\{i\}}(\pi_{\{i\}}^S(\ttt))$ does not depend on the choice of the vector $\ell^{(i)}$ in the definition of $\pi_{\{i\}}^S$. Hence we may assume $\ell^{(i)}_j=0$. Then $\pi_{\{j\}}^{S\cup\{i\}}\pi_{\{i\}}^S= \pi_{\{i,j\}}^S$, and again by induction, $g^{S \cup\{i,j\}}(\pi_{\{i,j\}}^S(\ttt))$ does not depend on the particular choice of vectors in the definition of $\pi_{\{i,j\}}^S$. Hence 
$$
\partial_j \left(g^{S \cup\{i\}}(\pi_{\{i\}}^S(\ttt))\right) = \partial_i \left(g^{S \cup\{j\}}(\pi_{\{j\}}^S(\ttt))\right)= g^{S \cup\{i,j\}}(\pi_{\{i,j\}}^S(\ttt)), 
$$
which proves the claim. 

Next we prove $L_S \subseteq L_{g^{S}}$. Let $\yy \in L_S$ and $i \in V_S$. Since $\partial_i g^S(\ttt)= g^{S \cup\{i\}}(\pi_{\{i\}}^S(\ttt))$ for all $i \in V_S$, it follows by induction, \rev{(\ref{incl}), and Euler's formula} that 
\rev{
\[
    \partial^\alpha g^S(\yy) = \frac{1}{d-|S|-|\alpha|} \sum_{i \in V_S} t_i \cdot \partial^\alpha g^{S \cup \{i\}}(\pi_{\{i\}}^S(\yy)) = 0
\]
}
unless $|\alpha|=d-|S|$. By Taylor expanding $g^S(\ttt+\yy)$, 
$$
g^S(\ttt+\yy) =\sum_{\alpha} \frac{\partial^\alpha g^S(\yy)}{\alpha!} \ttt^\alpha,
$$
we see that $g^S(\ttt+\yy)=g^S(\ttt)$ as claimed. This finishes the proof of the theorem for the weakly hereditary case by induction. 
Finally, since $L \subseteq L_f$, if $(\Delta,L)$ is hereditary, then $f$ is hereditary.
\end{proof}

\begin{lemma}\label{dirprod}
Let $\Delta_1$ and $\Delta_2$ be pure simplicial complexes on disjoint sets. Suppose $f_1$ is the unique weakly hereditary polynomial corresponding to $(\Delta_1$, $L_1$, $w_1$),  and $f_2$ is the unique weakly hereditary polynomial corresponding to ($\Delta_2$, $L_2$, $w_2$) (given by Theorem \ref{uniquepoly}). If $f_1,f_2$ are hereditary, then $f_1f_2$ is the unique hereditary polynomial corresponding to $(\Delta_1\times \Delta_2$, $L_1\oplus L_2$, $w_1w_2$). 

\end{lemma}
\begin{proof}
    The fact that $f_1f_2$ is hereditary follows from $\Delta_{f_1f_2} = \Delta_1 \times \Delta_2$ and $L_1 \oplus L_2 \subseteq L_{f_1f_2}$. Further, $\partial^{S_1 \cup S_2} (f_1f_2) = \partial^{S_1} f_1 \partial^{S_2} f_2 = w_1(S_1) w_2(S_2)$ for all facets $S_1 \in \Delta_1$ and $S_2 \in \Delta_2$ (and equals 0 otherwise), which shows that $f_1f_2$ is a weakly hereditary polynomial corresponding to $\Delta_1\times \Delta_2$, $L_1\oplus L_2$ and $w_1w_2$. The uniqueness then follows from Theorem \ref{uniquepoly}.
\end{proof}

\begin{proposition}\label{prodHL}
Suppose $f$ and $g$ are hereditary polynomials on disjoint sets of variables. If $f$ and $g$ are positive and hereditary Lorentzian, then so is $fg$. 
\end{proposition}

\begin{proof}
We apply Theorem \ref{mainthm_hereditary}. 
Clearly $\Delta_{fg}= \Delta_f \times \Delta_g$, $L_{fg}=L_f \oplus L_g$ and $\CCC_{fg}=\CCC_f \times \CCC_g$. Hence $fg$ is hereditary and $\CCC_{fg}$ is non-empty. The Cartesian product of two non-empty simplicial complexes is automatically connected. Hence, the H-connectedness of $\Delta_{fg}$ follows from the H-connectedness of $\Delta_f$ and $\Delta_g$. This verifies (C).

The quadratics of $fg$ appearing in (Q) are either products of two linear polynomials with positive coefficients, or are positive constants times 
the quadratics appearing in (Q) for $f$ or $g$. This verifies (Q) for $fg$.
\end{proof}

\section{Volume polynomials of matroids}\label{matroidsec}
In \cite{AHK}, Adiprasito, Huh and Katz proved that the Chow ring of a matroid satisfies the K\"ahler package. As a consequence the volume polynomial of a matroid is Lorentzian with respect to the cone of strictly submodular elements. A simplified proof of this consequence using Hodge theory and Lorentzian polynomials was proved in \cite{BES}. We will now give a self contained proof of this fact based  on the theory developed in Section \ref{hersec}. This leads to an elementary proof of the Heron-Rota-Welsh conjecture first proved in \cite{AHK}. 

Let $\LL$ be the lattice of flats of a matroid \rev{of rank $r(\LL)$} on a finite set $E$, i.e., $\LL$ is a collection of subsets of $E$, ordered by inclusion, satisfying 
\begin{itemize}
\item if $F,G \in \LL$, then $F \cap G \in \LL$, and 
\item for each $F$ in $\LL$, $\{G\setminus F\}_{F \prec G}$ partitions $E \setminus F$.\footnote{Here $F\prec G$ means $G$ covers $F$ in $\LL$.} 
\end{itemize}
 
The \emph{order complex}, $\Delta(\LL)$, of $\LL$ is the simplicial complex of all chains (totally ordered subsets) of 
$\underline{\LL}=\LL\setminus \{\zero, \one\}$, where $\zero$ denotes the smallest element in $\LL$, i.e., the set of loops of the matroid.  Let further 
$L(\LL)$ be the linear subspace of $\RR^{\underline{\LL}}$ consisting of all \emph{modular} 
$(y_F)_{F \in \underline{\LL}}$, i.e.,  there exists real numbers $c_i$, $i \in \one \setminus \zero$, such that for each $F \in \underline{\LL}$,
$$
y_F= \sum_{i \in F \setminus \zero}c_i \ \ \ \mbox{ and } \ \ \ 
\sum_{i \in \one \setminus \zero}c_i=0.
$$
We shall  apply Theorem \ref{uniquepoly} with $\Delta(\LL)$, $L(\LL)$ and $w(F)=1$ for all facets $F$ of $\Delta(\LL)$, to construct a hereditary polynomial $\pol_\LL \in \RR[t_G : G \in \underline{\LL}]$. To this end, we first show that $(\Delta(\LL),L(\LL))$ is hereditary. Given a chain of flats 
\[
   \zero= F_0  < F_1 < F_2 < \cdots < F_m < F_{m+1}=E,
\]
and a vector $(x_j)_{j=1}^m \in \RR^m$, we want to find a vector $(c_i)_{i \in \one \setminus \zero}$  such that 
$$
x_j= \sum_{i \in F_j\setminus F_0}c_i, \ \ \ 1\leq j \leq m, \ \ \mbox{ and } \sum_{i \in E\setminus K}c_i=0.
$$
Clearly this can be done since $F_0 \subset F_1 \subset F_2 \subset \cdots \subset F_m \subset  F_{m+1}$.

Next define $w(S) = 1$ for all facets $S \in \Delta(\LL)$. Let $S=\{F_1,\ldots, F_m\}$ be a facet of $\tau\Delta(\LL)$, i.e., $F_j \prec F_{j+1}$ for all $j$ except for exactly one value $0\leq k\leq m$. Then
\[
    \sum_{G \not\in S} w(S \cup \{G\}) t_G = \sum_{F_k \prec G \prec F_{k+1}} t_G.
\]
Since the interval $[F_k,F_{k+1}]$ is the lattice of flats of a rank-2 matroid, we have that $V_S = \{G: F_k \prec G \prec F_{k+1}\}$ partitions $F_{k+1} \setminus F_k$. Thus for any $\yy \in L(\LL)_S$,
\[
    \sum_{F_k \prec G \prec F_{k+1}} y_G = \sum_{i \in F_{k+1} \setminus F_k} c_i = \sum_{i \in F_{k+1}} c_i - \sum_{i \in F_k} c_i = 0,
\]
by the definition of $L(\LL)_S$. 
Therefore Theorem \ref{uniquepoly} applies, so that there is a unique hereditary polynomial $\pol_{\LL}$ of degree \rev{$d(\LL) = r(\LL)-1$} satisfying $\Delta_{\pol_{\LL}} = \Delta(\LL)$, $L(\LL) \subseteq L_{\pol_{\LL}}$, and $\partial^S \pol_{\LL} = 1$ for all facets $S$ of $\Delta(\LL)$. \rev{Notice that $\pol_{\LL}$ is positive}. 

Further, for all $S=\{F_1,\ldots, F_m\} \in \Delta(\LL)$, 
\begin{equation}\label{f-def}
 \Delta_{\pol_\LL^S} = \lk_{\Delta(\LL)}(S) = \prod_{i=0}^m \Delta([F_i,F_{i+1}]) \ \mbox{ and }  \  \pol_\LL^S(\ttt) = \prod_{i=0}^m \pol_{[F_i,F_{i+1}]}(\ttt), 
\end{equation}
by Lemma \ref{dirprod}. 

\begin{example} \label{lowdeg-matroid}
    When $d(\LL)=r(\LL)-1=1$, it is clear from Theorem \ref{uniquepoly} that
    \[
        \pol_{\LL}(\ttt) = \sum_{F \in \underline{\LL}} t_F.
    \]
    When $d(\LL)=2$, we claim that 
    $$
2 \cdot \pol_\LL(\ttt) = \left(\sum_{F}t_F\right)^2- \sum_{G} \left(t_G-\sum_{F<G}t_F\right)^2, 
$$
where $F$ denotes the flats of rank one, and $G$ denotes the flats of rank two. Since the set of rank one flats partitions $E\setminus K$, it follows that $L(\LL)$ is contained in the lineality space of the polynomial, $f$, on the right hand side. Moreover $\Delta_f= \Delta(\LL)$. Also the coefficients in front of the linear terms $t_Ft_G$ of $f$ are all equal to two. Hence the claim follows from Theorem \ref{uniquepoly}.
\end{example}

Given a matroid on $E$ with lattice of flats $\LL$, a vector $(y_S)_{\zero \subseteq S \subseteq \one}$ of real numbers is called \emph{strictly submodular} if $y_K = y_E = 0$ and
\[
    y_S + y_T > y_{S \cap T} + y_{S \cup T}
\]
for all incomparable sets $S$ and $T$. \rev{Notice that $S$ in $(y_S)_{\zero \subseteq S \subseteq \one}$ runs over all subsets of $E$ which contain $K$, not just flats.} We denote by $\SMOD_\LL$ the open convex cone in $\RR^{\underline{\LL}}$ obtained by restricting strictly submodular vectors to the entries which are labelled  by flats in $\underline{\LL}$, i.e., $(y_F)_{F \in \underline{\LL}}$. Notice that $L(\LL)$ is the lineality space of $\SMOD_\LL$.

\begin{lemma}\label{smod-effective}
    $\SMOD_\LL \subseteq \CCC_{\pol_\LL}$.
\end{lemma}
\begin{proof}
We first prove that $\SMOD_\LL$ is effective.  Notice that $\vv = (|F \setminus \zero| \cdot |\one \setminus F|)_{\zero < F < \one}$ is a vector in $\SMOD_\LL$ with positive entries. Suppose $\yy \in \SMOD_\LL$. Then $\zz := \yy - \epsilon \vv \in \SMOD_\LL$ for $\epsilon > 0$ sufficiently small. By e.g. \cite[Prop.~4.4]{Murota}, there exists $\ww \in L(\LL)$ such that $\zz + \ww \in \RR_{\geq 0}^{\underline{\LL}}$. Thus $\yy + \ww = \zz + \ww + \epsilon\vv \in \RR_{> 0}^{\underline{\LL}}$ as desired.

\rev{
We now prove $\SMOD_\LL \subseteq \CCC_{\pol_\LL}$ by induction on the degree of $\pol_\LL$, where $\LL = [\zero, \one]$. First, the inclusion is clear when $\pol_\LL$ is linear. To complete the proof, we need to prove 
$$\pi_{\{F\}}(\yy) \in \CCC_{\pol_\LL^{\{F\}}},$$ for all $F \in \underline{\LL}$ and $\yy \in \SMOD_\LL$. Recall that $\pol_\LL^{\{F\}} = \pol_{[\zero,F]} \cdot \pol_{[F,\one]}$. Further, $\pi_{\{F\}}(\yy) \in \SMOD_{[\zero,F]} \oplus \SMOD_{[F,\one]}$ by the definition of $\pi_{\{F\}}$, since the vectors in $L(\LL)$ are modular. Thus by induction $\pi_{\{F\}}(\yy) \in \SMOD_{[\zero,F]} \oplus \SMOD_{[F,\one]} \subseteq \CCC_{\pol_{[\zero,F]}} \oplus \CCC_{\pol_{[F,\one]}} = \CCC_{\pol_\LL^{\{F\}}}$, as desired.
}
%
 %
\end{proof}

\begin{theorem}\label{matroid-hered-Lor}
    \rev{If $\LL$ is the lattice of flats of a matroid, then the polynomial $\pol_\LL$ is hereditary Lorentzian.}
\end{theorem}
\begin{proof}
By Lemma \ref{smod-effective}, $\CCC_{\pol_\LL}$ is nonempty. By Remark \ref{main-remark}, Proposition \ref{prodHL}, \eqref{f-def} and induction on the rank of $\LL$, it suffices to prove that $\Delta(\LL)$ is connected for any $\LL$, and that the Hessian of $\pol_{\LL}(\ttt)$ has at most one positive eigenvalue whenever the rank of $\LL$ is three. The latter follows immediately from Example \ref{lowdeg-matroid}.

Recall that lattices of flats of matroids are graded and semimodular, see \cite[Chapter 1.7]{Oxley}. Also, by \cite[Prop. 3.3.2]{stanley} a finite lattice $\LL$ is semimodular if and only if for all $a,b \in \LL$,
\[
    \text{if $a$ and $b$ cover $a \wedge b$, then $a \vee b$ covers $a$ and $b$.}
\]
Connectedness of $\Delta(\LL)$ now follows from semimodularity.
\end{proof}

\subsection{Heron-Rota-Welsh conjecture}\label{lc}

The \emph{characteristic polynomial} of a matroid $\MM$ with lattice of flats $\LL = [\zero,\one]$ is 
\begin{equation}\label{chara}
\chi_{\MM}(t) = \sum_{F \in \LL} \mu(\zero, F) t^{r([F,\one])}, 
\end{equation}
\rev{where $\mu$ is the M\"obius function of $\LL$, see \cite[Section 3.7]{stanley}.}
If $\MM$ has rank at least one, then $\chi_\MM(t)$ is divisible by $t-1$, see \cite[Section 7]{White87}. The  \emph{reduced characteristic polynomial}  of a matroid $\MM$   is then 
$
\overline{\chi}_\MM(t) =  {\chi_\MM(t)}/(t-1)
$.  The next theorem, which was first proved by Adiprasito, Huh and Katz \cite{AHK}, solved the Heron-Rota-Welsh conjecture. For completeness we provide an alternative  proof below based on our approach.
\begin{theorem}[\cite{AHK}]\label{HRWconj}
The absolute values of the coefficients of the reduced characteristic polynomial of a matroid form a log-concave sequence. 
\end{theorem}

Consider the following two elements in the closure of $\SMOD_{\LL}$
$$\alpha_\LL= \left(\frac {|F\setminus \zero|}{|\one\setminus \zero|} \right)_{\zero < F < \one} \mbox{ and } \ \ \beta_\LL = \left( \frac {|\one\setminus F|}{|\one\setminus \zero| } \right)_{\zero < F < \one}.
$$
\rev{For those familiar with the techniques of \cite{AHK}, we point out that the polynomial $\pol_\LL$ is the volume polynomial of the Chow ring of the matroid associated to the geometric lattice $\LL$ (see \cite[Sec. 5]{BL}), and the elements $\alpha_\LL$ and $\beta_\LL$ of the closure of $\SMOD_{\LL}$ correspond to the classes $\alpha$ and $\beta$ as defined in \cite{AHK}. 
}

Throughout this section we fix an element $i \in \one\setminus \zero$, and let $\alpha_{\LL,i}=(a_F)_{\zero < F < \one}$ and $\beta_{\LL,i}=(b_F)_{\zero < F < \one}$ be the $0/1$-valued vectors defined by 
$a_F=1$ if and only if $i \in F$ and $b_F=1$ if and only if $i \not \in F$.
Then 
\begin{equation}\label{modeq}
\alpha_\LL-\alpha_{\LL,i} \in L(\LL) \ \ \mbox{ and } \ \ \beta_\LL-\beta_{\LL,i} \in L(\LL),
\end{equation}
for all $i \in \one \setminus \zero$. 

The elements $\alpha_\LL$ and $\beta_\LL$ behave well under the projections $\pi_S$. We leave the proof to the reader. 

\begin{lemma}\label{alpha-beta}
If $\LL = [\zero,\one]$ and  $\zero<F<\one$, then
\[
    \pi_{\{F\}}(\alpha_\LL) = (0, \alpha_{[F,\one]}) \ \ \ \mbox{and} \ \ \ \pi_{\{F\}}(\beta_\LL) = (\beta_{[\zero,F]}, 0).
\]
\end{lemma}

\rev{Recall that $d(\LL)$ denotes the degree of the polynomial $\pol_\LL$ and $r(\LL)$ denotes the rank of $\LL$, so that $d(\LL) = r(\LL) - 1$.}

\begin{lemma}\label{volalpha}
If  $\LL = [\zero,\one]$ is the lattice of flats of a matroid, then 
$$
\ff_\LL(\alpha_\LL) = \frac 1 {d(\LL)!}.
$$
\end{lemma}
\begin{proof}
Since $\{ A\setminus \zero\}_{\zero\prec A \leq \one}$ partitions $\one \setminus \zero$, there is a unique $H \in \underline{\LL}$ containing $i$ for which $\zero \prec H < \one$. By Euler's formula and \eqref{modeq},
$$
d(\LL) \cdot \ff_\LL(\alpha_\LL)= d(\LL) \cdot \ff_\LL(\alpha_{\LL,i})= 
\sum_{H\leq F < E} \partial_{t_F} \ff_\LL(\alpha_{\LL,i}).
$$
By \eqref{modeq} and \eqref{derproj} the latter sum may be written as
$$
\sum_{H\leq F < E}\pol_\LL^{\{F\}}(\pi_{\{F\}}(\alpha_\LL)), 
$$
and by \eqref{f-def} and Lemma~\ref{alpha-beta}, this sum may be expressed as  
$$
\sum_{H\leq F < E} \ff_{[\zero,F]} ( 0)\cdot  \ff_{[F,\one]} (\alpha_{[F,\one]}) = \ff_{[H,\one]}(\alpha_{[H,\one]}),  
$$
since $\ff_{[\zero,F]} (0)=0$ unless $F$ is an atom, and then $\ff_{[\zero,F]} (0)=1$. The lemma now follows by induction over $d(\LL)$. 
\end{proof}

 The next theorem is usually stated as a consequence of Weisner's theorem, see \cite[p. 277]{stanley}. 
\begin{theorem}[Weisner's theorem]
If $x\prec a \leq y$ are elements in a semimodular lattice $L$,  then 
$
\mu(x,y) = -\sum_b\mu(x,b) 
$
where the sum is over all $b \in L$ for which $x\leq b\prec y$ and $a \not \leq b$. 
\end{theorem}
A consequence of Weisner's theorem is that the M\"obius function of a semimodular lattice $\LL$ alternates in sign, i.e., 
$(-1)^{\rho(b)-\rho(a)} \mu(a,b) \geq 0$, where $\rho$ is the rank function of $\LL$. 

\begin{lemma}\label{volbeta}
If $\LL = [\zero,\one]$ is the lattice of flats of a matroid, then
$$
\ff_\LL(\beta_\LL)= \frac {|\mu(\zero,\one)|}{d(\LL)!}. 
$$
\end{lemma}
\begin{proof}
By \eqref{f-def}, \eqref{modeq}, and Lemma~\ref{alpha-beta}, 
$$
d(\LL) \cdot \ff_\LL(\beta_\LL)= \sum_{\zero<F<\one} b_F  \cdot \ff_{[\zero,F]}(\beta_{[\zero,F]}) \cdot  \ff_{[F,\one]}(0)= 
\sum_{\stackrel{\zero< F \prec \one}{F \not \ni i}} \ff_{[\zero,F]}(\beta_{[\zero,F]}).
$$
The lemma then follows by induction and Weisner's theorem. 
\end{proof}

\begin{theorem}\label{char_comp}
Suppose $\LL = [\zero,\one]$ is the lattice of flats of a matroid. Then 
$$
d(\LL)! \cdot \ff_\LL(s\alpha_\LL +t\beta_\LL) = \sum_{\stackrel {\zero \leq F < \one} {i \not \in F}} \binom {d(\LL)}{r([\zero,F])}  \cdot |\mu(\zero,F)| \cdot t^{r([\zero,F])} s^{d([F,\one])}.
$$
\end{theorem}

\begin{proof}
Let $f(t)= \ff_\LL(s\alpha_\LL +t\beta_\LL)$. \rev{Then
$$
f'(t) = D_{\beta_\LL}\ff_\LL(s\alpha_\LL +t\beta_\LL)= \sum_{\stackrel {\zero < F < \one} {i \not \in F}} \ff_\LL (s \pi_{\{F\}}(\alpha_\LL) +t\pi_{\{F\}}(\beta_\LL)).
$$}
Hence by Lemma \ref{alpha-beta} and \eqref{modeq},
\[
    f'(t) = \sum_{\stackrel {\zero < F < \one} {i \not \in F}} \ff_{[\zero,F]}(\beta_{[\zero,F]}) \cdot  \ff_{[F,\one]}(\alpha_{[F,\one]}) \cdot t^{d([\zero,F])} s^{d([F,\one])}.
\]
By Lemmas \ref{volalpha} and \ref{volbeta}, 
$$
f'(t)= \sum_{\stackrel {\zero < F < \one} {i \not \in F}}   |\mu(\zero,F)| \cdot \frac {t^{d([\zero,F])}} {d([\zero,F])!} \cdot \frac {s^{d([F,\one])}}{d([F,\one])!}. 
$$
The lemma follows since 
$
f(0)=  s^{d(\LL)} /{d(\LL)!}
$,
by Lemma \ref{volalpha}.
\end{proof}

\begin{lemma}[Cor. 7.27, \cite{White87}]\label{red}
Suppose $\LL = [\zero,\one]$ is the lattice of flats of a matroid $\MM$. Then 
\begin{equation}\label{rcf}
\overline{\chi}_{\MM}(t) = \sum_{F \not \ni i} \mu(\zero,F) t^{d([F,\one])}. 
\end{equation}
\end{lemma}

\begin{proof}[Proof of Theorem \ref{HRWconj}]
If $\alpha=\alpha_\LL$, $\beta=\beta_\LL$ and $d=d(\LL)$, then 
$$
d! \cdot \ff_\LL(s\alpha_\LL +t\beta_\LL)= (sD_{\alpha}+tD_{\beta})^d \ff_\LL= \sum_{k=0}^d \binom d k \cdot (D_\alpha^k D_{\beta}^{d-k}\ff_\LL) \cdot s^kt^{d-k}.  
$$
 Hence by Theorem \ref{char_comp} and Lemma \ref{red}, the coefficient in front of  $t^k$ in the reduced characteristic polynomial of $\MM$ is equal to $D_\alpha^k D_{\beta}^{d-k}\ff_\LL$. The theorem now follows from Theorem \ref{matroid-hered-Lor} and Remark \ref{AF-remark}.
\end{proof}

\section{Stellar subdivisions of hereditary polynomials}\label{subdsec}
In this section we study stellar subdivisions of hereditary polynomials. These are linear operators which for volume polynomials of Chow rings of simplicial fans correspond to stellar subdivisions of the fans. \rev{The main result of this section is Theorem~\ref{support-general}, that subdivisions and their inverses preserve the hereditary Lorentzian property. This is analogous to \cite[Thm. 1.6]{Lagrangian} on the K\"ahler package of Chow rings of simplicial fans.}

Let $S \in \Delta$, where $|S|\geq 1$. The \emph{stellar subdivision}, $\Delta_S$,  of $\Delta$ on $S$ is the simplicial complex on $V \cup \{0\}$, where $0 \notin V$, obtained by 
\begin{itemize}
\item removing all faces containing $S$, and 
\item adding all faces $R \cup \{0\}$, where $S \not \subseteq R$ and $R \cup S \in \Delta$.
\end{itemize}
For positive real numbers $\ccc =(c_i)_{i\in S}$, let 
$$
L^\ccc= \left\{ (\ell_0, \ell) \in \RR\times \RR^V : \ell \in L \mbox{ and } \ell_0 = \sum_{i \in S}c_i\ell_i\right\}.
$$
\begin{lemma}\label{simplp}
If $(\Delta,L)$ is hereditary, then so is $(\Delta_S, L^\ccc)$ \rev{for every $S$ and $\ccc$}. 
\end{lemma} 

\begin{proof}
It suffices to prove \eqref{simpl} for facets $F$ of $\Delta_S$. 

If $0\not \in F$, then $F \in \Delta$ so that  \eqref{simpl} for $(\Delta_S, L^\ccc)$ follows from \eqref{simpl} \rev{for} $(\Delta, L)$. Otherwise $0 \in F$, and then $F$ contains all but one element of $S$, and $(F \setminus \{0\}) \cup S \in \Delta$. In this case \eqref{simpl} for $F$ follows from \eqref{simpl} for $(F \setminus \{0\}) \cup S \in \Delta$. 
\end{proof}

\begin{lemma}\label{switchlemma}
 Suppose $(\Delta,L)$ is hereditary, where $\Delta$ has dimension $d-1$.  If $S \in \Delta$ and $g \in \PPP^d(\Delta_S,L^\ccc)$, then 
 $$
 \partial_0^k\partial^T \partial_i g/c_i =  \partial_0^k\partial^T \partial_j g/c_j= -\partial_0^{k+1}\partial^Tg 
 $$
 for all $T \subseteq V$ such that $|T|=d-k-1$, $k>0$, and $i,j \in S \setminus T$. 
\end{lemma}
\begin{proof}
The lemma is trivially true if  $T\cup\{0\} \notin \Delta_S$. Suppose $T\cup\{0\} \in \Delta_S$. The polynomial $\partial_0^k g^T$ is hereditary by Lemma~\ref{herbydiff}. Hence 
$$
\partial_0^{k+1}\partial^T g \cdot t_0 + \sum_{i \in S\setminus T}  \partial_0^k\partial^T \partial_i g\cdot t_i =0, 
$$
for all $\ttt \in \RR^{S \cup \{0\}}$ such that $t_0 = \sum_{i\in S\setminus T} c_it_i$. The lemma follows.  
\end{proof}
Define a linear operator $\bus_S^\ccc : \RR[t_i : i \in V\cup\{0\}] \longrightarrow  \RR[t_i : i \in V]$ by 
$$
\bus_S^\ccc(g) = g\big|_{t_0= \sum_{i\in S}c_it_i}. 
$$
\begin{lemma}\label{gfcoef}
 Suppose $(\Delta,L)$ is hereditary where $\Delta$ has dimension $d-1$, and that $S \in \Delta$. Then 
 $$ 
\bus_S^\ccc : \PPP^d(\Delta_S,L^\ccc) \longrightarrow \PPP^d(\Delta,L),  
$$
is injective. 

Moreover if $T$ is a facet of $\Delta$ and $f=\bus_S^\ccc(g)$, then 
\begin{align}
 \partial^T f &= c_j\partial_0 \partial^{T\setminus\{j\}}g,  \ \   \mbox{ if $j\in S \subseteq T$, and} \label{fg1} \\ 
 \partial^T f &= \partial^{T}g,  \ \   \mbox{ otherwise}. \label{fg2}
\end{align}
\end{lemma}
\begin{proof}
Let $g \in \PPP^d(\Delta_S, L^\ccc)$ and $f=\bus_S^\ccc(g)$. Then $L \subseteq L_f$. Suppose $T \notin \Delta$. If $S \cap T=\varnothing$, then $\partial^T f \equiv 0$. If $T=S_0 \cup R$, where $S_0=T \cap S \neq \varnothing$, then $\{0\} \cup R \notin \Delta_S$. By the chain rule, $\partial^Tf \equiv 0$. Hence $f \in \PPP^d(\Delta,L)$. 

Let $T=S'\cup R$, where $S'=T\cap S$, be a facet of $\Delta$. Then, by the chain rule,  
$$
\partial^T f = \prod_{i \in S'}(c_i\partial_0+\partial_i)\partial^R g.
$$
If $|S'|=0$, then $\partial^T f= \partial^Tg$. If $|S'|>0$, then by using Lemma \ref{switchlemma},
$$
\partial^T f = \prod_{i \in S'}(c_i\partial_0+\partial_i)\partial^R g = (c_j\partial_0+\partial_j)\partial^{S'\setminus\{j\}} \partial^Rg,
$$
where $j \in S'$. If $S'=S$, then $\partial^T f= c_j\partial_0 \partial^{T\setminus\{j\}}g$. If $S' \neq S$, then $\partial^T f = \partial^T g$. This proves \eqref{fg1}
and \eqref{fg2}. 

From Theorem \ref{uniquepoly} and \eqref{fg1}, \eqref{fg2} it follows that $\bus_S^\ccc$ is injective. 
\end{proof}

The next theorem introduces the \emph{stellar subdivision operator}\footnote{Thanks to Chris Eur for asking if subdivisions may be generalized beyond fans.} $\sub_S^\ccc$. Notice that $\sub_S^\ccc$ is universal in the sense that the definition only depends on $S$ and $\ccc$, and not on $\Delta$ and $L$. 

\begin{theorem}\label{subgen}
 Suppose $(\Delta,L)$ is hereditary where $\Delta$ has dimension $d-1$, and that $S \in \Delta$. Then $\bus_S^\ccc : \PPP^d(\Delta_S,L^\ccc) \longrightarrow \PPP^d(\Delta,L)$ is bijective, and its inverse is given by the linear operator $\sub_S^\ccc : \PPP^d(\Delta,L) \longrightarrow \PPP^d(\Delta_S,L^\ccc)$ defined by 
$$
\sub_S^\ccc(f) =f - (-1)^{s} \sum_{n=s}^\infty \frac{z^n}{n!} \cdot h_{n-s}(\dc) \, \dc^S f, \ \ \ \ \mbox{ where } s=|S|, 
$$
and where $h_k(\dc)$ is the complete homogeneous symmetric polynomial of degree $k$ in the variables $\dc_i=\partial_i/c_i$, $i \in S$, and $z = t_0 - \sum_{i \in S} c_i t_i$.

Moreover, $f$ is positive if and only if $\sub_S^\ccc(f)$ is positive. 
\end{theorem}

\begin{proof}
We first prove  $\sub_S^\ccc : \PPP^d(\Delta,L) \longrightarrow \PPP^d(\Delta_S,L^\ccc)$. Let $f \in \PPP^d(\Delta,L)$, and let $g=\sub_S^\ccc(f)$. It is straightforward to see that  $L^\ccc \subseteq L_g$. We claim $\Delta_g \subseteq \Delta_S$. By the definition of $\Delta_S$, for any $T \subseteq V \cup \{0\}$,
$$
(S\cup T) \setminus \{0\} \in \Delta \ \ \Longrightarrow \ \ T \in \Delta_S \mbox{ or } S \subseteq T.
$$
Assume $T \notin \Delta_S$ and $S \not \subseteq T$. Then $(S \cup T) \setminus \{0\} \notin \Delta$, so that $\partial^{S\cup T}f \equiv 0$, from which it follows $\partial^Tg = \partial^Tf$. If $0 \in T$, then $\partial^Tf \equiv 0$. If $0 \notin T$, then $T \notin \Delta$ so that $\partial^Tf \equiv 0$. Hence $T \not \in \Delta_g$. To prove the claim it remains to prove $\partial^Sg \equiv 0$. If $e_j(\dc)$ denotes the elementary symmetric polynomial of degree $j$ in the variables $\dc_i$, $i \in S$, then
\[
\begin{split}
    \dc^Sg  &= \dc^S f - (-1)^{s} \sum_{n=s}^\infty\sum_{j=0}^{s}  (-1)^{s-j} \frac{z^{n-s+j}}{(n-s+j)!} \cdot e_{j}(\dc)\, h_{n-s}(\dc) \, \dc^S f \\
        &= \dc^S f - \sum_{n=0}^\infty \frac{z^{n}}{n!} \left(\sum_{j=0}^{n} (-1)^{j} e_{j}(\dc) \, h_{n-j}(\dc)\right) \dc^S f =0,      \end{split}
\]
since $\sum_{j=0}^{n}  (-1)^{j} e_{j} \, h_{n-j}=\delta_{0n}$. 

Since $(\bus_S^\ccc \circ \sub_S^\ccc)(f) =f$ for all $f$ by definition, $\sub_S^\ccc$ is a bijection. The final statement follows from Lemma~\ref{gfcoef}.
\end{proof}

\rev{We now observe various basic connections between a hereditary polynomial and its stellar subdivisions, including descriptions of links and H-connectivity.}

\begin{lemma}\label{sub-induct-rules}
    Let $f \in \PPP^d(\Delta,L)$ and $S \in \Delta$. For $i \in V$, 
    \begin{align*}
        \sub_S^\ccc(f)^{\{i\}} &= \sub_S^\ccc(f^{\{i\}}), & & \mbox{if } i \not\in S \\
        \sub_S^\ccc(f)^{\{i\}} &= \sub_{S\setminus \{i\}}^{\ccc'}(f^{\{i\}}), & & \mbox{if } i \in S
    \end{align*}
    where $\ccc'$ is equal to $\ccc$ restricted to $S \setminus \{i\}$. Further if $S = \{i\}$, then 
    \[
        \sub_S^\ccc(f) = \left.f\right|_{t_i = t_0/c_i}.
    \]
\end{lemma}
\begin{proof}
    The case when $i \not\in S$ follows immediately from the definition of $\sub_S^\ccc$. The case when $S = \{i\}$ follows from the fact that $\bus_S$ is the inverse of $\sub_S$ and that $\{i\} \not\in \Delta_S$. For the case of $i \in S$, we compute
    \[
    \begin{split}
        \dc_i \sub_S^\ccc(f) &= \dc_i f - (-1)^{s-1} \sum_{n=s-1}^\infty \frac{z^n}{n!} \left[h_{n+1-s}(\dc) - h_{n-s}(\dc) \dc_i\right] \dc^S f \\
            &= \dc_i f - (-1)^{s-1} \sum_{n=s-1}^\infty \frac{z^n}{n!} \cdot h_{n-(s-1)}\big((\dc_j)_{j \in S \setminus \{i\}}\big) \, \dc^{S \setminus\{i\}} \dc_i f
    \end{split}
    \]
    where $h_{-1} = 0$. The result then follows by scaling by $c_i$ and setting $t_i = 0$.
\end{proof}

The following lemma contains standard facts, but also follows as a corollary of Lemma~\ref{sub-induct-rules}.

\begin{lemma} \label{sub-induct-rules-delta}
   Let $\Delta$ be a simplicial complex on $V$,  and let $S \in \Delta$. For $i \in V$ we have
    \begin{align*}
        \lk_{\Delta_S}(\{i\}) &= (\lk_{\Delta}(\{i\}))_S, & & \mbox{if } i \not\in S, \\
        \lk_{\Delta_S}(\{i\}) &= (\lk_{\Delta}(\{i\}))_{S\setminus\{i\}}, & & \mbox{if } i \in S, \\
      \lk_{\Delta_S}(\{0\}) &= \{R \subseteq V: S \not \subseteq R \mbox{ and } R\cup S \in \Delta\}. & & 
\end{align*}
\end{lemma}

\begin{lemma} \label{H-conn_equiv}
    Given a simplicial complex $\Delta$ and $S \in \Delta$, $\Delta$ is H-connected if and only if $\Delta_S$ is H-connected.
\end{lemma}
\begin{proof}
    We may assume that $|S| \geq 2$. It is straightforward to see that $\lk_{\Delta_S}(\{0\})$ is always connected. 

    By Lemma~\ref{sub-induct-rules-delta} and induction on the dimension of $\Delta$, we now only need to show that $\Delta$ is connected if and only if $\Delta_S$ is connected, which is standard and clear by definition. 
\end{proof}

Define an equivalence relation on positive hereditary polynomials  by $f \sim g$ if there exists a finite sequence of hereditary polynomials 
$$
f=f_0, f_1, \ldots, f_m=g
$$
such that for each $1\leq k \leq m$, either 
$$
\sub_S^\ccc(f_{k-1})= f_k \ \ \mbox{ or }\ \  \sub_S^\ccc(f_{k})= f_{k-1},
$$
for some positive real numbers $\ccc$ and $S \in \Delta_{f_{k-1}}$ or $S \in \Delta_{f_{k}}$. 

\begin{theorem}\label{support-general}
Suppose $f \sim g$, and that $\CCC_f$ and $\CCC_g$ are nonempty. Then $f$ is hereditary Lorentzian if and only if $g$ is hereditary Lorentzian. 
\end{theorem}

\begin{proof}
Consider the property (Q) restricted to positive hereditary polynomials $f$ of degree $d$: 
\begin{itemize}
    \item[(Q)] The Hessian of $f^T$ has at most one positive eigenvalue for each $T \in \Delta_f$ of size $d-2$. 
\end{itemize}
Let $g = \sub_S^\ccc(f)$. We prove by induction over $d$ that $f$ satisfies (Q) if and only if $g$ satisfies (Q).
If $s=1$, then this follows from Lemma \ref{sub-induct-rules}, and if $d < s$, then this follows from the fact that $g = f$. Assume $2 \leq s \leq d$.

If $d=2$, then 
$
g = f-z^2 \dc^S f/2
$. 
If $f$ satisfies (Q), then so does $g$ by the Cauchy interlacing theorem, since $\dc^S f > 0$. If $g$ satisfies (Q), then so does $f=\bus_S^\ccc(g)$ by Sylvester's law of inertia and Cauchy interlacing.  This verifies the statement for $d=2$.

If $d=3$, then $g^{\{0\}}$ is either identically zero or the product of two linearly independent linear polynomials. Hence $g^{\{0\}}$ has at most one positive eigenvalue. The statement for $d=3$ follows from Lemma \ref{sub-induct-rules} for $T = \{i\}$, $i \neq 0$, by induction.

Finally if $d \geq 4$, then the statement follows from Lemma \ref{sub-induct-rules} by induction.

Suppose $f \sim g$, and that $\CCC_f$ and $\CCC_g$ are nonempty. By Lemma \ref{H-conn_equiv},  $\Delta_f$ is H-connected if and only if $\Delta_g$ is H-connected.  Also, $f$ satisfies (Q) if and only if $g$ satisfies (Q). The theorem now follows from Theorem \ref{mainthm_hereditary}. 
\end{proof}

The following proposition may be proven by induction.  
\begin{proposition}\label{consub}
Let $f \in \RR[t_i : i \in V]$ be a positive hereditary polynomial, and suppose $S \in \Delta_f$. 
If $\vv \in \KKK_f$, then  $\left(\sum_{i \in S}c_iv_i - \epsilon, \vv\right) \in \KKK_{\sub_S^\ccc(f)}$ for all sufficiently small  $\epsilon>0$. 
\end{proposition}

Hence if $\KKK_f$ is nonempty, then $f$ is hereditary Lorentzian if and only if $\sub_S(f)$ is hereditary Lorentzian.

\section{Volume polynomials of simple polytopes}\label{polytsec}

We now turn to the classical topic of volume polynomials associated to simple polytopes and convex bodies. These polynomials are known to be $\CCC$-Lorentzian by e.g. Hodge theory or convex geometry  \cite{Aleksandrov, McMullen, stanleyp,Timorin}. We prove that such polynomials are Lorentzian using the theory developed in Section \ref{hersec}. This leads to a self-contained and new proof of the Alexandrov-Fenchel inequalities \cite{Aleksandrov}. 

We refer to \cite{Schneider,Timorin} for undefined terminology.  
In this section  $P$  will denote  a full dimensional simple polytope of dimension $d$ in a Euclidean space $(\VV,\langle \cdot, \cdot \rangle)$. Let $\rho_1, \ldots, \rho_n$ be the (outward) unit normals of the facets of $P$.  Let further $\CCC_P$ be the set of all simple polytopes in $\VV$ that have the same normals as $P$. A polytope $Q \in \CCC_P$ is uniquely determined by its \emph{support numbers} 
$$
t_i(Q) = \max_{q \in Q} \langle \rho_i, q \rangle, \ \ \ \ 1 \leq i \leq n.
$$
If the origin is in the interior of $Q$, then $t_i(Q)$ is the distance to the supporting hyperplane in direction $\rho_i$. 
Since  \rev{the Minkowski sum $\lambda Q +  \mu R$  lies in  $\CCC_P$ whenever $Q, R \in \CCC_P$ and $\lambda, \mu >0$,} and 
$$
t_i(\lambda Q + \mu R) = \lambda t_i(Q) + \mu t_i(R), \ \ \ \lambda, \mu >0, 
$$ 
we may identify $\CCC_P$ with an open convex cone in $\RR^n$. The volume on $\CCC_P$ defines a polynomial in $t_i(Q)$, $1\leq i \leq n$, i.e., 
$$
\Vol(Q)= \pol_P(t_1(Q), \ldots,  t_n(Q)),  \ \ \ \mbox{ for all } Q \in \CCC_P,
$$
for a unique polynomial $\pol_P \in \RR[t_1,\ldots, t_n]$. 

Consider the linear space 
$$
L_P= \left\{ \Big(\langle \rho_1, \yy\rangle, \ldots, \langle \rho_n, \yy\rangle\Big) : \yy \in \VV\right\} \subseteq \RR^n.
$$
Then $\pol_P(\ttt+\ww)=\pol_P(\xx)$, for all $\ww \in L_P$, $\ttt \in \CCC_P$. Since $\CCC_P$ is open, $L_P$ is contained in the lineality space of $\pol_P$.    

Let $\Delta_P$ be the simplicial complex whose simplices are $\{ i \in [n]  :  F \subseteq P_i \}$, where $F$ is any face of $P$, and $P_i$ is the facet of $P$ with normal $\rho_i$.

For details of the proof of the next proposition we refer to \cite{Schneider,Timorin}.  
\begin{proposition}\label{derpolt}
Let $P$ be a simple polytope, and let $P_i$ be the facet  of $P$ with normal $\rho_i$ (translated to the orthogonal complement $\VV_i$ of $\rho_i$ in $\VV$). Then 
 \[
    \partial_i \pol_P(\ttt) = \pol_{P_i}\left(\left(\frac{t_j - t_i \cos(\theta_{ij})} {\sin(\theta_{ij})}\right)_j\right), 
\]
where $\theta_{ij}$ is the angle between $\rho_i$ and $\rho_j$.
\end{proposition}

\begin{proof}
Let $\ttt \in \CCC_P$ be the support numbers of a polytope $Q$ which contains the origin in its interior, and is such that the ray along $\rho_i$ intersects the relative interior of the facet $Q_i$. It is straightforward that $\partial_i \pol_P(\ttt) = \vol(Q_i)$. The polytope $Q' = Q - t_i \rho_i$ contains the origin in the relative interior of $Q_i'$.  The support numbers of $Q'$ are then 
\[
    t_j' = t_j - t_i \langle\rho_i, \rho_j\rangle =t_j - t_i \cos(\theta_{ij}). 
\]
Suppose $P_j$ is adjacent to $P_i$. The projection of the unit normal vector $\rho_j$ onto $\VV_i$ 
is given by $\sin(\theta_{ij}) \cdot \rho_j'$, where $\rho_j' \in \VV_i$ is the corresponding unit normal vector of $P_i$. This implies the $1/\sin(\theta_{ij})$ factors. 
\end{proof}

\begin{lemma}\label{herpol}
Let $P$ be a simple polytope. Then $(\Delta_P, L_P)$ is hereditary, and $\CCC_P \subseteq \CCC_{\pol_P}$. 
    
\end{lemma}
\begin{proof}
 If $S \in \Delta_P$, then $\rho_i,  i \in S$ are linearly independent since the dual of $P$ is simplicial. Hence $(\Delta_P, L_P)$ is hereditary. Let $S \in \Delta_P$ and $Q \in \CCC_P$, and let $F$ be the  face of $Q$ corresponding to $S$. Then we may translate by $\yy \in \VV$ so that the origin is in the relative interior of $F$. Then $t_i(Q+\yy)=0$ for all $i \in S$ and $t_j(Q+\yy)>0$ for all $j$ in the link of $S$. Hence $Q \in \CCC_{\pol_P}$.
\end{proof}

From Proposition \ref{derpolt} we deduce

\begin{corollary}\label{heredP}
Let $P$ be a simple $d$-dimensional polytope. Then $\pol_P \in \PPP^d(\Delta_P,L_P)$ is hereditary, and $\Delta_{\pol_P}=\Delta_P$. Moreover $\pol_P$ is positive. 
\end{corollary}

\begin{lemma} \label{simplex_pol}
    If $P$ is a $d$-dimensional simplex containing $0$ in its interior, then $\PPP^d(\Delta_P,L_P)$ is one-dimensional and contains a nonzero polynomial of the form
    \[
        f(\ttt) = (v_1 t_1 + \cdots + v_{d+1} t_{d+1})^d
    \]
    for some positive constants $v_1,\ldots,v_{d+1}$.
\end{lemma}
\begin{proof}
Since $P$ is a simplex, there is a (unique up to positive scalar) $\vv \in \RR_{>0}^{d+1}$ such that $v_1 \rho_1 + \cdots + v_{d+1} \rho_{d+1} = 0$. Clearly $f(\ttt)$ for this $\vv$ is in $\PPP^d(\Delta_P,L_P)$. On the other hand if $g \in \PPP^d(\Delta_P,L_P)$, then $L_g$ contains $\{ \ttt \in \RR^{d+1} : v_1 t_1 + \cdots + v_{d+1} t_{d+1}=0\}$, from which it follows that $g$ is a constant multiple of $f$. 
\end{proof}

\begin{theorem}\label{mainP}
If $P$ is a simple polytope, then $\pol_P$ is hereditary Lorentzian. 
\end{theorem}

\begin{proof}
The cone $\CCC_{\vol_P}$ is nonempty by Lemma \ref{herpol}. Moreover $\Delta_P$ is H-connected. By Proposition \ref{derpolt}, it remains to prove that $\vol_P$ has at most one positive eigenvalue for $d=2$. By performing suitable edge subdivisions and inverse edge subdivisions to $\vol_P$, it follows that  $\vol_P \sim f$ for some polynomial $f$ such that  $\Delta_f = \Delta_Q$ and $L_Q \subseteq L_f$, where $Q$ is a triangle. Since $\PPP^2(\Delta_Q,L_Q)$ is one-dimensional by Lemma \ref{simplex_pol}, it follows that $f$ is a positive constant multiple of $\vol_Q$. The theorem now follows from Lemma \ref{simplex_pol} and Theorem \ref{support-general}.
\end{proof}

As a corollary of Theorem \ref{mainP}, we obtain the Alexandrov-Fenchel inequalities for convex bodies. The \emph{mixed volume} of convex bodies   $K_1, K_2, \ldots, K_n$ in $\RR^n$ may be defined by 
$$
V(K_1,K_2,\ldots, K_n)= \partial_1\partial_2 \cdots \partial_n \Vol(t_1K_1+t_2K_2+\cdots+t_nK_n).
$$

\begin{corollary}[Alexandrov-Fenchel inequalities \cite{Aleksandrov}]
    Let $K_1,K_2,\ldots,K_n$ be convex bodies in $\RR^n$. Then 
    \[
        V(K_1,K_2,K_3,\ldots,K_n)^2 \geq V(K_1,K_1,K_3,\ldots,K_n) \cdot V(K_2,K_2,K_3,\ldots,K_n).
    \]
    
\end{corollary}
\begin{proof}
If $K_1, \ldots, K_n$ are simple polytopes in $\CCC_P$, \rev{then using Proposition \ref{comp}, }the Alexandrov-Fenchel inequalities follow directly from Theorem \ref{mainP} and Remark \ref{AF-remark}. By a simple approximation argument, the general case reduces to this case, see  \cite[Theorem 2.4.15]{Schneider}. 
\end{proof}

The same proof also reproves a recent result of Huh and the first author \cite{BH}, namely that the polynomial
$$
\Vol(t_1K_1+t_2K_2+\cdots+t_nK_n)
$$
is $\RR_{>0}^n$-Lorentzian, whenever $K_1, K_2, \ldots, K_n$ are convex bodies in $\RR^n$.

\section{Chow rings of simplicial fans}\label{chowsec}
We will now apply the theory developed so far to Chow rings of simplicial fans. The methods apply to more general rings which are intimately connected to hereditary polynomials. 

Let $\Delta$ be a simplicial complex on $V$, and let $L$ be a linear subspace of $\RR^V$. 
Define a graded $\RR$-algebra by 
$$
A(\Delta,L)= \bigoplus_{k \geq 0}A^k(\Delta,L)=\frac {\RR[x_i : i \in V]}{I(\Delta)+J(L)},
$$
where 
\begin{itemize}
    \item $I(\Delta)$ is the Stanley-Reisner ideal of $\Delta$, i.e., the ideal in $\RR[x_i : i \in V]$ generated by all monomials $\rev{\xx^S :=} x_{s_1}\cdots x_{s_k}$,  for which $\rev{S =} \{s_1,\ldots, s_k\}$ is not an face of  $\Delta$, and 
    \item $J(L)$ is the ideal generated by $\sum_{i \in V}\ell_i x_i$, for all $\ell \in L$. 
\end{itemize}
Let $\Sigma$ be a simplicial fan in a finite dimensional real vector space $\VV$ with specified ray vectors $\rho_i$, $i \in V$, which generate the one-dimensional cones of $\Sigma$. Let $L(\Sigma)$ be the linear subspace of $\RR^V$ defined by
$$
L(\Sigma) = \left( \lambda(\rho_i) : \lambda \in \VV^{\ast} \right)_{i \in V}, 
$$
and let $\Delta(\Sigma)$ be the simplicial complex on $V$ defined by
\[
    \Delta(\Sigma) = \{S \subseteq V : \{\rho_i : i \in S \} \mbox{ generates a cone in } \Sigma\}.
\]

\begin{definition}\label{chow}
The \emph{Chow ring} of a simplicial fan $\Sigma$ is the graded $\RR$-algebra $A(\Sigma)= A(\Delta(\Sigma), L(\Sigma))$. 
\end{definition}
Notice that we don't require that $\Sigma$ is rational, as is commonly required.

\begin{lemma}\label{herfan}
    If $\Sigma$ is a simplicial fan, then $(\Delta(\Sigma),L(\Sigma))$ is hereditary.
\end{lemma}
\begin{proof}
    Follows from the fact that any $k$-dimensional cone in $\Sigma$ is spanned by $k$ linearly independent ray vectors.
\end{proof}

For  $\alpha \in A^k(\Delta,L)^*$, let $\hat{\alpha}$ be the polynomial in $\RR[t_i: i \in V]$ defined by
$$
\hat{\alpha}(\ttt)= \frac 1 {k!} \alpha\left(\left(\sum_{i \in V}t_ix_i \right)^k\right).
$$
Then 
\begin{equation}\label{hatt1}
\partial_j \hat{\alpha}(\ttt)= \frac 1 {(k-1)!} \alpha\left(x_j\left(\sum_{i \in V}t_ix_i \right)^{k-1}\right), 
\end{equation}
from which it follows that $\hat{\alpha}(\ttt) \in \PPP^k(\Delta,L)$.
\begin{proposition}\label{charPK}
The linear map $A^k(\Delta,L)^\ast \longrightarrow \PPP^k(\Delta,L)$ defined by 
$
\alpha \longmapsto \hat{\alpha}(\ttt)
$
is  bijective. 
\end{proposition}

\begin{proof}
If $a \in I(\Delta)+J(L)$ and $f \in \PPP^k(\Delta,L)$, then $a(\partial)f=0$. Hence the linear map
$
\phi : \PPP^k(\Delta,L) \to A^k(\Delta, L)^*
$
is well-defined by 
$
\phi(f)(a)= a(\partial)f.
$
The inverse of $\phi$ is given by 
$
\alpha \longmapsto \hat{\alpha}
$.
\end{proof}

\begin{lemma}\label{multlink}
Suppose $\Delta$ is a simplicial complex on $V$, and that $L$ is a linear subspace of $\RR^V$. If $S \in \Delta$, then 
$$
a \longmapsto \xx^S a, 
$$
defines a linear map $\phi_S: A^{k}(\lk_\Delta(S), L_S) \longrightarrow A^{k+|S|}(\Delta, L)$, where 
$$
L_S= \{(\ell_i)_{i \in V_S}: (\ell_i)_{i \in V} \in L \mbox{ and } \ell_j=0 \mbox{ for all } j \in S\}.
$$
Moreover, if $\alpha \in A^{k+|S|}(\Delta,L)^*$, then
\begin{equation}\label{hatt2}
\widehat{\alpha \circ \phi_S}= \hat{\alpha}^S. 
\end{equation}
\end{lemma}
\begin{proof}
If $\xx^T \in I(\lk_\Delta(S))$, then $\xx^S \xx^T \in I(\Delta)$. If $(\ell_i)_{i \in V_S} \in L_S$, then 
$$
\xx^S \sum_{i \in V_S}\ell_i x_i = \xx^S \sum_{i \in V}\ell_i x_i - \sum_{j : S\cup \{j\} \notin \Delta} \ell_jx_j\xx^S \in I(\Delta) + J(\Delta).
$$
Hence $\phi_S$ is well-defined. Moreover, \eqref{hatt2} follows from \eqref{hatt1}. 
\end{proof}

Let $\sigma$ be a cone in the simplicial fan $\Sigma$, corresponding to the face $S \in \Delta(\Sigma)$. The \emph{star} of $\sigma$ in $\Sigma$ is the fan in $\VV/\spn(\sigma)$ with cones 
$$
\st_{\Sigma}(\sigma)= \{ \tau+\spn(\sigma) : \tau \in \Sigma \mbox{ contains } \sigma \}.
$$
It is straightforward to see that 
$$
\Delta(\st_{\Sigma}(\sigma))=\lk_{\Delta(\Sigma)}(S) \ \ \mbox{ and } \ \ L(\st_{\Sigma}(\sigma))= L(\Sigma)_S.  
$$
Hence Lemma \ref{multlink} identifies a canonical map $A^d(\Sigma)^* \longrightarrow A^{d-\dim(\sigma)}(\st_{\Sigma}(\sigma))^*$. 

\begin{lemma}\label{amo}
Suppose $(\Delta, L)$ is hereditary.  Any $a \in A(\Delta, L)$ may be written as 
$$
a= \sum_{S \in \Delta} \mu(S) \prod_{i \in S} x_i, \ \ \ \ \mu(S) \in \RR.
$$
In particular, $A^m(\Delta,L) =(0)$ for all $m>d$, where $d-1$ is the dimension of $\Delta$. 
\end{lemma}

\begin{proof}
Suppose  
$$
x_{i_1}^{\alpha_1} \cdots x_{i_s}^{\alpha_s} \in A^k(\Delta, L)\setminus \{0\}, 
$$
where $\alpha_j>0$ for all $1\leq j \leq s$. 
Then $S=\{i_1,\ldots, i_s\} \in \Delta$. If $s=k$ we have nothing to prove. Hence we may assume $\alpha_j>1$ for some $j$. Without loss of generality we may assume $j=1$ and $i_1=1$. There is an element $\ell \in L$ such that $\ell_1=1$ and $\ell_i=0$ for all $i \in S \setminus \{1\}$.  Then 
$$
x_{i_1}^{\alpha_1} \cdots x_{i_s}^{\alpha_s}= \left(x_1-\sum_{i \in V}\ell_i x_i\right)x_{i_1}^{\alpha_1-1}x_{i_2}^{\alpha_2} \cdots x_{i_s}^{\alpha_s}.
$$
By induction over $\sum_{j=1}^s(\alpha_j-1)$ this proves the lemma. 
\end{proof}

The following convex cone is the one commonly used for fans, called the \emph{ample cone} or the cone of \emph{strictly convex elements}. 
\begin{definition}\label{chowcone}
Let $\Delta$ be a $(d-1)$-dimensional simplicial complex on $V$, and let $L$ be a linear subspace of $\RR^V$. Suppose $(\Delta,L)$ is weakly hereditary. Define an open convex cone $\CCC(\Delta,L)$ recursively as follows. 
\begin{itemize}
\item[(1)] If $d=1$, then $\CCC(\Delta,L)= \RR_{>0}^V+L$,
\item[(2)] If $d>1$, then 
$$
\CCC(\Delta,L) = (\RR_{>0}^V +L)\cap \{ \vv \in \RR^V : \pi_{\{i\}}(\vv) \in \CCC(\lk_{\Delta}(\{i\}),L_{\{i\}}) \mbox{ for all } i \in V\}. 
$$
\end{itemize}
If $\Sigma$ is a simplicial fan we define $\CCC(\Sigma)= \CCC(\Delta(\Sigma), L(\Sigma))$. 
\end{definition}

\begin{lemma}
    If $(\Delta,L)$ is weakly hereditary and $\Delta' \subseteq \Delta$ and $L \subseteq L'$, then $(\Delta',L')$ is weakly hereditary and $\CCC(\Delta,L) \subseteq \CCC(\Delta',L')$.
\end{lemma}
\begin{proof}
    \rev{Both statements follow from the fact that $(\Delta', L')$ being weakly hereditary is weaker than $(\Delta, L)$ being weakly hereditary.}
\end{proof}

\begin{corollary} \label{subfan-cone}
    Let $\Sigma$ be a simplicial fan and suppose $\CCC(\Sigma)$ is nonempty. Then $\CCC(\Sigma')$ is nonempty for any subfan $\Sigma'$ of $\Sigma$.
\end{corollary}

We say that a functional $\alpha \in A^d(\Delta,L)^*$ is hereditary Lorentzian if $\hat{\alpha}$ is hereditary Lorentzian. Moreover $\alpha$ is called \emph{positive} if $\alpha(\xx^F) \geq 0$ for all facets $F$ of $\Delta$. It follows that if $\alpha$ is positive and $(\Delta, L)$ is weakly hereditary, then $ \CCC(\Delta,L) \subseteq \CCC_{\hat{\alpha}}$.

\begin{example}
    Let $f_\theta(\ttt) = (t_1+t_2+t_3)^2-\theta (t_1+t_2-t_0)^2$. This polynomial is hereditary for all $\theta \in \RR$, hereditary Lorentzian for $\theta > 0$, and has a negative linear coefficient for $\theta > 1$. Further, $f_\theta$ is not a fan polynomial for $0 < \theta < 1$. To see this, fix $0 < \theta < 1$ and suppose $\Sigma$ is a quadratic fan with ray vectors $\rho_0,\rho_1,\rho_2,\rho_3 \in \VV$ for which $f_\theta = \hat\alpha$ for some $\alpha \in A^2(\Sigma)^*$. Then $L(\Sigma) \subseteq L_{f_\theta}$ is at least two-dimensional, and thus 
    \[
        L(\Sigma) = L_{f_\theta} = \{(\ell_i)_{i=0}^3 \in \RR^4 : \ell_1 + \ell_2 + \ell_3 = 0 \mbox{ and } \ell_1 + \ell_2 - \ell_0 = 0\}.
    \]
    The condition $\ell_1 + \ell_2 - \ell_0 = 0$ implies $\rho_0 = \rho_1 + \rho_2$, and thus $\rho_0$ lies in the relative interior of the cone generated by $\rho_1$ and $\rho_2$. However $\{1,2\} \in \Delta_{f_\theta} \subseteq \Delta(\Sigma)$, which contradicts that $\Sigma$ is a fan.
\end{example}

The following direct consequence of Theorem \ref{mainthm_hereditary} gives a characterization of hereditary Lorentzian functionals on Chow rings of fans. 

\begin{theorem}\label{genfanchar}
Let $\Sigma$ be a simplicial fan, and let $\alpha \in A^k(\Sigma)^*$, $1\leq k \leq d$. If $\CCC_{\hat\alpha}$ is nonempty, then $\alpha$ is hereditary Lorentzian if and only if 
\begin{itemize}
\item[(C)] $\tau\Delta_{\hat\alpha}$ is H-connected, and
\item[(Q)] for all $S \in \Delta_{\hat\alpha}$ with $|S|=k-2$, the Hessian of $\hat{\alpha}^S$ has at most one positive eigenvalue. 
\end{itemize}
\end{theorem}

We also formulate a version for $d$-dimensional fans for which $A^d(\Sigma)$ is one-dimensional. The \emph{volume polynomial} of $\Sigma$  is then defined (up to a  constant) as $\vol_\Sigma= \hat{\alpha}$, where \rev{$\alpha \in A^d(\Sigma)^*$} is chosen so that $\alpha(\xx^F)>0$ for some facet $F$ of $\Delta(\Sigma)$. We say that $\Sigma$ is \emph{positive} if $\hat{\alpha}$ is positive. We say that $\Sigma$ is \emph{hereditary Lorentzian} if $\vol_{\Sigma}$ is hereditary Lorentzian. 

\begin{theorem}\label{posfanchar}
Let $\Sigma$ be a positive $d$-dimensional simplicial fan. If $\CCC(\Sigma)$ is nonempty, then $\Sigma$ is hereditary Lorentzian if and only if 
\begin{itemize}
\item[(C)] $\Delta(\Sigma)$ is H-connected, and
\item[(Q)] for all $S \in \Delta(\Sigma)$ with $|S|=d-2$, the Hessian of $\vol_{\Sigma}^S$ has at most one positive eigenvalue. 
\end{itemize}
\end{theorem}

\begin{remark}
If $A^2(\st_\Sigma(\sigma))$ is one-dimensional for all cones  $\sigma$ of dimension $d-2$, then (Q) is equivalent to \rev{the statement} that $\st_\Sigma(\sigma)$ is (hereditary) Lorentzian for all cones  $\sigma$ of dimension $d-2$. 
\end{remark}

\begin{remark}
In Section 5 of \cite{BaHu}, Babaee and Huh  constructed a two-dimensional positive simplicial fan in $\RR^4$ for which the Hessian of its volume polynomial has more than one positive eigenvalue. Hence this fan fails to be hereditary Lorentzian. We thank June Huh for pointing out this example. 
    
\end{remark}

The \emph{product} of two fans $\Sigma_1$ and $\Sigma_2$ on $\VV_1$ and $\VV_2$, respectively is the fan $\Sigma_1 \times \Sigma_2$ on $\VV_1\oplus \VV_2$ with cones $\sigma_1 \times \sigma_2$, where $\sigma_1 \in \Sigma_1$ and $\sigma_2 \in \Sigma_2$.  The following theorem is a direct consequence of Lemma \ref{dirprod}. 

\begin{theorem}
Suppose $\Sigma = \Sigma_1 \times \Sigma_2$ is the product of two simplicial fans of dimensions $p$ and $q$, respectively. If $\alpha \in A^p(\Sigma_1)^*$ and $\beta \in A^q(\Sigma_2)^*$ are positive hereditary Lorentzian, then so is $\alpha \times \beta \in A^{p+q}(\Sigma)^*$. 

\end{theorem}

\subsection{A support theorem for Lorentzian functionals on fans}\label{suppfans}
Recall that the support $|\Sigma|$ of a fan $\Sigma$ is the union of its cones. We will prove here that the Lorentzian property only depends on the support of the fan. More precisely, if $\Sigma$ and $\Sigma'$ are $d$-dimensional simplicial fans with the same support, then there is a canonical bijection $A^d(\Sigma)^* \longrightarrow A^d(\Sigma')^*$,      $\alpha \longmapsto \alpha'$, such that 
$$
 \mbox{$\alpha$ is hereditary Lorentzian if and only if $\alpha'$  is hereditary Lorentzian, }
$$
provided that $\CCC_{\hat\alpha}$ and $\CCC_{\hat\alpha'}$ are nonempty. The corresponding property for Lefschetz fans was recently proved by Ardila, Denham and Huh in \cite{Lagrangian}.

Let $\Sigma$ be a simplicial fan, and let $\rho_0$ be a vector in the relative interior of a cone $\sigma \in \Sigma$. The \emph{stellar subdivision} of $\Sigma$ with respect to  $\rho_0$ is the simplicial fan $\sub_{\rho_0}(\Sigma)$ obtained from $\Sigma$ by
\begin{itemize}
    \item removing all cones containing $\sigma$, and
    \item adding all cones generated by $\tau$ and $\rho_0$, where $\sigma \not \subseteq \tau \in \Sigma$ is such that the cone generated by $\tau$ and $\sigma$ is in $\Sigma$.
\end{itemize}
If $S = \{i \in V : \rho_i \in \sigma\} \in \Delta(\Sigma)$, then 
$$\rho_0 = \sum_{i \in S} c_i \rho_i$$ 
for \rev{some} positive real numbers $\ccc = (c_i)_{i \in S}$. It follows that $\Delta(\sub_{\rho_0}(\Sigma)) = \Delta(\Sigma)_S$ and $L(\sub_{\rho_0}(\Sigma)) = L(\Sigma)^\ccc$, and thus $A(\sub_{\rho_0}(\Sigma)) = A(\Delta(\Sigma)_S, L(\Sigma)^\ccc)$. Hence by Theorem \ref{subgen} and Proposition \ref{charPK}, $\sub_{\rho_0}$ defines a linear bijection 
$$
\sub_{\rho_0} : A^d(\Sigma)^* \longrightarrow A^d(\sub_{\rho_0}(\Sigma))^*. 
$$
Moreover $\sub_{\rho_0}$ preserves positivity. Clearly $|\Sigma| = |\sub_{\rho_0}(\Sigma)|$, and the converse is also true\footnote{Theorem A in \cite{Wlod} is formulated for rational fans. However it is true also for fans in $\RR^n$, see the comment after Theorem 8.1 in \cite{Wlod}.}.  

\begin{theorem}[Theorem A in \cite{Wlod}]\label{subsupp}
Let $\Sigma$ and $\Sigma'$ be simplicial fans in $\VV$. Then $|\Sigma|=|\Sigma'|$ if and only if there exists a finite sequence of simplicial fans
\[
    \Sigma = \Sigma_0, \Sigma_1, \ldots, \Sigma_m = \Sigma'
\]
such that for each $1 \leq k \leq m$, either
\[
    \sub_\rho(\Sigma_{k-1})= \Sigma_k \ \ \mbox{ or }\ \  \sub_\rho(\Sigma_{k})= \Sigma_{k-1}, 
\]
for some  $\rho$ in $|\Sigma_{k-1}|=|\Sigma_k|$.
    
\end{theorem}

 Although the sequence of subdivisions and welds (inverse subdivisions) described above may not be unique, we now show that the associated bijection $\phi_{\Sigma,\Sigma'}: A^d(\Sigma)^* \longrightarrow A^d(\Sigma')^*$ is.

\begin{theorem}\label{canonical-bijection}
    Let $\Sigma$ and $\Sigma'$ be $d$-dimensional simplicial fans with the same support. There is a canonical bijection $\phi_{\Sigma,\Sigma'}: A^d(\Sigma)^* \longrightarrow A^d(\Sigma')^*$,  given by any valid sequence of subdivisions and welds.
\end{theorem}
\begin{proof}
There exists a bijection $\phi: A^d(\Sigma)^* \longrightarrow A^d(\Sigma')^*$  given by a finite sequence of subdivisions and welds, by Theorems \ref{subgen} and \ref{subsupp}. It remains to prove that this bijection  does not depend on the specific sequence. 

Fix a Euclidean norm on $\VV$. For a cone $\sigma$ in a simplicial fan, let $V(\sigma)$ denote the (relative) volume of the parallelotope spanned by the ray vectors of $\sigma$. Suppose $\Sigma'= \st_{\rho_0}(\Sigma)$, and let $\alpha \in A^d(\Sigma)^*$ and $\alpha'=\sub_{\rho_0}(\alpha) \in  A^d(\Sigma')^*$. Let further $\sigma$ and $\sigma'$ be two maximal cones in $\Sigma$ and $\Sigma'$, respectively, such that $\sigma \cap \sigma'$ has nonempty interior. If $F$ and $F'$ are the faces corresponding to $\sigma$ and $\sigma'$, then 
$$
V(\sigma) \alpha(\xx^F) = V(\sigma') \alpha'(\xx^{F'}),
$$
follows from Lemma \ref{gfcoef}. 

Now let $\phi: A^d(\Sigma)^* \longrightarrow A^d(\Sigma')^*$ be any linear bijection given by a finite sequence of subdivisions and welds, and let $\alpha \in A^d(\Sigma)^*$ and $\alpha'=\phi(\alpha) \in  A^d(\Sigma')^*$. If the intersection of $\sigma$ and $\sigma'$ has nonempty relative interior, it follows from iterating the argument above that 
$
V(\sigma) \alpha(\xx^F) = V(\sigma') \alpha'(\xx^{F'}).
$
Since $\widehat{\alpha'}$ is hereditary, $\alpha'$ is determined by the values on $\xx^F$, where $F$ ranges over all facets of $\Delta(\Sigma')$, by Theorem \ref{uniquepoly}. Hence $\phi$ does not depend on the choices of subdivisions and welds. 
\end{proof}

As a corollary of Theorems \ref{subgen} and \ref{support-general}, we deduce that the hereditary Lorentzian property only depends on the support of fans.

\begin{theorem}\label{support-fans}
    Let $\Sigma$ and $\Sigma'$ be $d$-dimensional simplicial fans with the same support. Let further $\alpha \in A^d(\Sigma)^*$ and $\alpha' \in A^d(\Sigma')^*$ be related by $\alpha' = \phi_{\Sigma,\Sigma'}(\alpha)$. If $\CCC_{\hat\alpha}$ and $\CCC_{\hat\alpha'}$ are nonempty, then $\alpha$ is hereditary Lorentzian if and only if $\alpha'$ is hereditary Lorentzian.
\end{theorem}

\begin{example}[The normal fan of a simple polytope]
    Recall the notation from Section \ref{polytsec}. Let $P$ be a full dimensional simple polytope in a $d$-dimensional Euclidean space $(\VV,\langle \cdot, \cdot \rangle)$, and let $\rho_1,\ldots,\rho_n$ be the (outward) unit normals of the facets of $P$. The \emph{normal fan}  $\Sigma_P$ of $P$ may be defined as the simplicial fan in $\VV$ with cones generated by $\{\rho_i\}_{i \in S}$ for all $S \in \Delta_P$. 
    Hence $A(\Sigma_P)=A(\Sigma_P, L_P)$. 
    
    The fan $\Sigma_P$ is complete, which means $|\Sigma_P| = \VV$, and thus $|\Sigma_P| = |\Sigma_Q|$ for any full dimensional simple polytope $Q$ in $\VV$. When $Q$ is a simplex, Lemma \ref{simplex_pol} implies $A^d(\Sigma_Q)^*$ is one-dimensional and is spanned by a single hereditary Lorentzian functional. By Theorem \ref{canonical-bijection}, $A^d(\Sigma_P)^*$ is thus defined by a single hereditary functional corresponding to  $\pol_P$, which is hereditary Lorentzian by Theorem \ref{mainP}.
\end{example}

\begin{example}[The Bergman fan of a matroid] \label{ex-bergman}
    Recall the notation from Section~\ref{matroidsec}. Let $\LL=[K,E]$ be the lattice of flats of a rank $d+1$ \rev{matroid}. For each proper flat $F \in \underline{\LL} = \LL \setminus \{\zero,\one\}$, define the vector $\rho_F \in \RR^{\one \setminus \zero} / \RR\mathbf{1}$, where $\mathbf{1}$ is the all ones vector, via
    \[
        \rho_F = \sum_{i \in F \setminus \zero} e_i + \RR\mathbf{1}
    \]
    where $e_i$ is a standard basis vector. The \emph{Bergman fan} $\Sigma_\LL$ is the subfan of the normal fan of the permutohedron with cones generated by $\{\rho_F\}_{F \in S}$ for all $S$ in the order complex $\Delta(\LL)$ (see e.g. \cite[Section 2.2]{BES}). It follows that $L(\Sigma_\LL)=L(\LL)$. 
    
    Since the permutohedron is a simple polytope, $\CCC(\Sigma_\LL)$ is nonempty by Corollary~\ref{subfan-cone}. Further, $\alpha \in (A^d(\Sigma_\LL))^*$ is uniquely defined up to scalar via $\alpha(\xx^S) = 1$ for every maximal chain $S \in \Delta(\LL)$ by \cite[Prop. 5.2]{AHK}. Hence $\vol_{\Sigma_\LL}=\pol_\LL$, and $\Sigma_\LL$ is positive. 
    H-connectedness of $\Delta(\LL)$ then follows from semimodularity of $\LL$ as in the proof of Theorem \ref{matroid-hered-Lor}, and the Hessian of $\vol_{\Sigma_\LL}^S$ has at most one positive eigenvalue for all $S \in \Delta(\LL)$ with $|S| = d-2$, by Example \ref{lowdeg-matroid}. Thus Theorem \ref{posfanchar} implies $\Sigma_\LL$ is hereditary Lorentzian.
\end{example}

\begin{example}[The conormal fan of a matroid] \label{ex-conormal}
    In \cite{Lagrangian} the authors define the \emph{conormal fan} of a matroid $\MM$, which they denote by $\Sigma_{\MM,\MM^\perp}$. They prove that $\Sigma_{\MM,\MM^\perp}$ is Lefschetz (see \cite[Def. 1.5]{Lagrangian}), which implies the Hodge-Riemann relations for $\Sigma_{\MM,\MM^\perp}$. From the Hodge-Riemann relations of degree 0 and 1, they then derive a number of long-standing conjectures regarding properties of $\MM$. Their proof that $\Sigma_{\MM,\MM^\perp}$ is Lefschetz uses \cite[Thm. 1.6]{Lagrangian}, which is a support theorem analogous to Theorem \ref{support-fans} for Lefschetz fans. Here we will use a similar proof strategy to prove the Hodge-Riemann relations of degree 0 and 1 for $\Sigma_{\MM,\MM^\perp}$ directly, by showing that $\Sigma_{\MM,\MM^\perp}$ is hereditary Lorentzian.
    
    First, $\Sigma_{\MM,\MM^\perp}$ is a subfan of the normal fan of a simple polytope called the bipermutohedron, by \cite[Props. 2.18 and 2.20]{Lagrangian}. Thus $\CCC(\Sigma_{\MM,\MM^\perp})$ is nonempty by Corollary~\ref{subfan-cone}. Further, $\Sigma_{\MM,\MM^\perp}$ has the same support as a product of Bergman fans of matroids by \cite[Section 3.4]{Lagrangian}. Products of Bergman fans are hereditary Lorentzian by Example \ref{ex-bergman} and Proposition \ref{prodHL}, and therefore $\Sigma_{\MM,\MM^\perp}$ is hereditary Lorentzian by Theorem \ref{support-fans}.
\end{example}

\section{Characterizations of Lorentzian polynomials on cones}\label{Lorsec}
In this section we provide various characterizations of Lorentzian polynomials on cones. We shall see how the Theorem \ref{mainthm_hereditary} yields the characterization of Lorentzian polynomials proved by June Huh and the first author in \cite{BH}. We first show that the notion of Lorentzian polynomials on cones extends the notion of Lorentzian polynomials in \cite{BH}. 

 \begin{definition}\label{L-def}
 A homogeneous polynomial $f \in \RR[t_1,\ldots, t_n]$ of degree $d\geq 2$ is called \emph{strictly Lorentzian} if  for all $1\leq i_1, \ldots, i_d \leq n$, 
 \begin{itemize}
\item[(P1)] $\partial_{i_1}  \cdots \partial_{i_d}f >0$, and 
\item[(H1)] the symmetric bilinear form 
 $$
(\xx, \yy) \longmapsto D_{\xx}D_{\yy} \partial_{i_3}\cdots \partial_{i_d} f
 $$
 is non-singular and has exactly one positive eigenvalue. 
 \end{itemize}
 Any positive constant or linear form with positive coefficients is strictly Lorentzian. Limits (in the Euclidean space of homogeneous degree $d$ polynomials) of strictly Lorentzian polynomials are called \emph{Lorentzian}. 
 \end{definition}

 \begin{proposition}\label{CL-equal}
 Let $f\in \RR[t_1,\ldots, t_n]$ be a homogeneous polynomial of degree $d\geq 0$. Then $f$ is $\RR_{>0}^n$-Lorentzian if and only if $f$ is Lorentzian.
 \end{proposition}
 
 \begin{proof}
\rev{Notice first} that the $\RR_{>0}^n$-Lorentzian and the Lorentzian property are closed properties\rev{, by Remark~\ref{closedc} and definition of Lorentzian above}. 
 
 We prove by induction on $d\geq 2$ that  any strictly Lorentzian polynomial is $\RR_{>0}^n$-Lorentzian, the case when $d<2$ being trivial. Suppose $f$ is strictly Lorentzian and of degree at least $3$. Then $\partial_if$ is strictly Lorentzian for each $i$. By induction, $\partial_if$ is $\RR_{>0}^n$-Lorentzian for each $i$. Theorem \ref{engine}, for the case when $L_\CCC=\{0\}$, now implies that $f$ is $\RR_{>0}^n$-Lorentzian. 
 
Suppose $f$ is in the interior of the space of $\RR_{>0}^n$-Lorentzian  polynomials. Then $f$ has positive coefficients and (H1) is satisfied, so that $f$ is strictly Lorentzian. 
 \end{proof}

We now associate to any homogeneous polynomial $f$, the \emph{polarization} $\Pi(f)$, which is a hereditary polynomial that reduces to $f$ by a change of variables. This polarization is different from the one considered in e.g. \cite{BB1,BH}. Suppose $f \in \RR[t_1, \ldots, t_n]$ is a $d$-homogeneous polynomial of degree $\kappa_i$ in $t_i$ for each $1\leq i \leq n$. Let $V_\kappa = \{ v_{ij} : 1\leq i \leq n \mbox{ and } 0\leq j \leq \kappa_i \}$, and define $\Pi(f) \in \RR[t_{ij}: v_{ij} \in V_\kappa]$  by 
$$
\Pi(f) = f(y_1, \ldots, y_n), \ \ \ \mbox{ where } y_i= \sum_{j=0}^{\kappa_i} t_{ij}. 
$$
For $S \subseteq V_\kappa$, let 
$$
\alpha(S) = \big( |S \cap \{v_{1j} : 0\leq j \leq \kappa_1 \}|, \ldots, |S \cap \{v_{nj} : 0\leq j \leq \kappa_n \}|\big). 
$$
Then $S$ is a facet of  $\Delta_{\Pi(f)}$ if and only if $\alpha(S) \in \supp(f)$, where 
$$
\supp(f) = \{ \alpha \in \NN^n : \alpha_1+\cdots+\alpha_n=d \mbox{ and } \partial^\alpha f \neq 0\}. 
$$
Clearly 
$$
\left\{ \ttt \in \RR^{V_\kappa} : \sum_{j=0}^{\kappa_i} t_{ij} =0 \mbox{ for all } 1\leq i \leq n \right\} \subseteq L_{\Pi(f)}, 
$$
from which it follows that $\Pi(f)$ is hereditary. 

\begin{proposition}\label{HL=L}
Let $f \in \RR[t_1,\ldots, t_n]$ be a homogeneous polynomial with nonnegative coefficients. Then $f$ is Lorentzian if and only if $\Pi(f)$ is hereditary Lorentzian. 
\end{proposition}

\begin{proof}
Suppose $\Pi(f)$ is hereditary Lorentzian. 
Let $P : \RR^{V_\kappa} \longrightarrow \RR^V$ be the linear projection defined by 
$$
P(\ttt)= \left(\sum_{j=0}^{\kappa_i} t_{ij}\right)_{i=1}^n.
$$
\rev{Let $\CCC= P^{-1}(\RR_{>0}^n)$. Then $\RR_{>0}^{V_\kappa} \subseteq \CCC \subseteq \CCC_{\Pi(f)}$}. Hence  $\Pi(f)$ is $\RR_{>0}^{V_\kappa}$-Lorentzian. Hence $f$ is Lorentzian by Propositions \ref{CL-equal} and \ref{comp}.

Suppose $f$ is Lorentzian. Then $\Pi(f)$ is $\CCC$-Lorentzian by Proposition \ref{comp}. Likewise  $\Pi(f)^S$ is $\pi_S(\CCC)$-Lorentzian for each $S \in \tau \Delta_{\Pi(f)}$. The theorem now follows from Proposition \ref{main-conv}.
\end{proof}

 For any nonempty $\MM \subset \NN^n$ and $\beta \in \NN^n$, define 
$$
\partial^\beta\MM= \{\alpha -\beta : \alpha \in \MM \mbox{ and } \alpha -\beta \in \NN^n\} \ \ \mbox{ and } \ \ \tau(\MM) = \!\!\!  \bigcup_{1 \leq j \leq n} \!\!\! \partial_j(\MM), 
$$
where $\partial_i = \partial^{e_i}$. We say that $\MM$ has \emph{constant sum} if there is a number $r$ such that 
$\alpha_1+\cdots+\alpha_n =r$ for all $\alpha \in \MM$.  
\rev{For $\MM$ with constant sum $r$,} we say that $\MM$ is \emph{connected} if $r \geq 2$ and there is no proper subset $A$ of $[n]$ such that 
$$
\MM = \{\alpha \in \MM : \alpha_i=0 \mbox{ for all } i \in A\} \cup \{\alpha \in \MM : \alpha_i=0 \mbox{ for all } i \in [n]\setminus A\}
$$
is a partition of $\MM$ into two nonempty sets. If $\MM$ has constant sum $r \leq 1$, we say that $\MM$ is connected by convention. If $\MM$ has constant sum $r \geq 2$, we say that $\MM$ is \emph{H-connected} if $\partial^\alpha \MM$ is connected for each $\alpha \in \NN^n$ with $|\alpha| \leq r-2$. 

\begin{theorem}\label{main-L2}
  Let $f\in \RR[t_1,\ldots, t_n]$ be a homogeneous polynomial of degree $d\geq 2$ with nonnegative coefficients. Then $f$ is Lorentzian if and only if 
\begin{itemize}
\item[(S2)] $\tau \supp(f)$ is \rev{H}-connected, and 
\item[(H1)] for any $1\leq i_3, \ldots, i_d \leq n$, the symmetric bilinear form 
$$
(\xx, \yy) \longmapsto D_{\xx}D_{\yy} \partial_{i_3}\cdots \partial_{i_d} f
$$
has at most one positive eigenvalue. 
\end{itemize}
 \end{theorem}
 
 \begin{proof}
We first note that (H1) holds for $f$ if and only if (Q) holds for $\Pi(f)$. Also, it is not hard to see that $\tau \Delta_{\Pi(f)}$ is H-connected if and only if $\tau \supp(f)$ is H-connected. The theorem thus follows from Theorem \ref{mainthm_hereditary} and Proposition \ref{HL=L}. 
\end{proof}

Let us recall the definition of an M-convex set.  
\begin{definition}\label{M-conv}
A finite subset  $\MM$ of  $\NN^n$ is \emph{M-convex} if the following \emph{exchange axiom} holds
\begin{itemize}
\item[(EA)] If $\alpha, \beta \in \MM$ and $\alpha_i > \beta_i$, then there is an index $j$ such that $\beta_j> \alpha_j$ and 
$$
\alpha-e_i+e_j \in \MM,
$$
where $e_i$ denotes the $i$th standard basis vector of $\RR^n$. 
\end{itemize}
\end{definition} 
A subset $\MM$ of  $\NN^n$  is M-convex if and only if it is the set of integer points of a generalized permutohedron, see \cite{postnikov}.
Notice that if $\MM$ is M-convex, then so is $\tau(\MM)$ and $\partial_i(\MM)$ for all $i$. Proposition \ref{alt-M} below is a converse to this.

 \begin{proposition}\label{alt-M}
Suppose $\MM \subset \NN^n$ has constant sum $r\geq 3$. If $\MM$ is connected and $\partial_i(\MM)$ are M-convex for all $i$, then 
$\MM$ is M-convex. 
\end{proposition}

\begin{proof}
First we show the exchange graph of $\MM$ is connected. If not, then let $\MM = S_1 \sqcup S_2$ be a partition such that $S_1$ and $S_2$ are disconnected in the exchange graph. By connectedness of $\MM$, there exist $\alpha \in S_1$, $\beta \in S_2$, and $k$ such that $\alpha_k,\beta_k > 0$. The M-convexity of $\partial_k(M)$ then implies $\alpha$ and $\beta$ are connected in the exchange graph, a contradiction. Therefore the exchange graph of $\MM$ is connected.

Given $\alpha, \beta \in \MM$ and $i$ such that $\alpha_i>\beta_i$, we will now show (EA); that is, that there exists $j$ such that $\beta_j>\alpha_j$ and
\begin{equation}\label{exch}
    \alpha-e_i+e_j \in \MM.
\end{equation}
We prove (\ref{exch}) by induction on distance in the exchange graph $d(\alpha,\beta)$. The base case is $d(\alpha,\beta) < r$. In this case $\alpha_k \beta_k \neq 0$ for some $k$, and thus (\ref{exch}) follows from the M-convexity of $\partial_k(M)$.

For the general case, let $A$ and $B$ be the support of $\alpha$ and $\beta$ respectively. If $A \cap B$ is non-empty, then (\ref{exch}) follows from the M-convexity of $\partial_k(\MM)$ as above. Otherwise fix a minimal-length path between $\alpha$ and $\beta$, and let $\gamma \in \MM$ be the element closest to $\alpha$ on the path with support which intersects $B$. Since there is a (unique) $k$ such that $\gamma_k\beta_k>0$, we may use the M-convexity of $\partial_k(\MM)$ to deduce that $d(\gamma,\beta)=r-1$. Thus $d(\alpha,\gamma) =  d(\alpha,\beta) - (r-1)$. By induction we apply (\ref{exch}) to $\gamma^{(0)} = \gamma$, $\alpha$ and any index $k \in B^c$ for which $\gamma_k > \alpha_k$, to obtain $\gamma^{(1)} \in M$. Iterating this process, we apply (\ref{exch}) by induction to $\gamma^{(s)}$, $\alpha$ and any index $k \in B^c$ for which $\gamma^{(s)}_k > \alpha_k$, to obtain $\gamma^{(s+1)} \in \MM$. We can do this for $s < r-1$ since $d(\alpha,\gamma^{(s)}) \leq d(\alpha,\gamma) + s$ for all $s$. Thus $\alpha' = \gamma^{(r-1)} \in \MM$ is such that $\alpha'_k \leq \alpha_k$ for all $k \in B^c$, which implies
\[
    \alpha' = \alpha - e_{k'} + e_{j_1}
\]
for some $k' \in A$ and $j_1 \in B$. If $k' = i$, then (\ref{exch}) is satisfied by $\alpha' \in \MM$. Otherwise, since $\alpha'_{j_1}, \beta_{j_1} > 0$, the M-convexity of $\partial_{j_1}(\MM)$ implies there exists
\[
    \alpha'' = \alpha' - e_i + e_{j_2} \in \MM
\]
where $j_2 \in B$. Thus
\[
    \alpha'' = \alpha - e_{k'} - e_i + e_{j_1} + e_{j_2} \in \MM.
\]
Since $r \geq 3$, there exists some $k$ such that $\alpha_k,\alpha''_k > 0$, and (\ref{exch}) follows from the M-convexity of $\partial_k(\MM)$ for some $j \in \{j_1,j_2\}$.
\end{proof}

\begin{lemma}\label{M-tM-connected}
    Suppose $\MM \subset \NN^n$ has constant sum $r \geq 3$. Then $\MM$ is connected if and only if the truncation $\tau\MM$ is connected.
\end{lemma}
\begin{proof}
    First, if $\MM$ is disconnected via a proper subset $A \subset [n]$ then $\tau\MM$ is also disconnected via $A$. On the other hand, if $\MM$ is connected then for any proper subset $A \subset [n]$, there exists $\alpha \in \MM$ which has support intersecting both $A$ and $[n] \setminus A$. Since $r \geq 3$, this implies there is some $\alpha' \in \tau\MM$ which has support intersecting both $A$ and $[n] \setminus A$. Therefore $\tau\MM$ is also connected.
\end{proof}

The following characterization of $\MM$-convex sets follows directly from Proposition \ref{alt-M} and Lemma \ref{M-tM-connected} by induction. 

\begin{proposition}\label{char-MC}
Suppose $\MM$ has constant sum $r \geq 3$. Then $\MM$ is M-convex if and only if 
\begin{itemize}
\item $\tau\MM$ is \rev{H}-connected, and 
\item $\partial^\alpha \MM$ is M-convex for each $\alpha \in \NN^n$ with $|\alpha|=r-2$. 
\end{itemize}
\end{proposition}
The next lemma is well-known, but we give a quick proof here for completeness.  
\begin{lemma}\label{quadr}
If $f$ is a Lorentzian quadratic, then $\supp(f)$ is M-convex. 
\end{lemma}
\begin{proof}
Let $H=(h_{ij})_{i,j=1}^n$ be the Hessian of $f$. Suppose there are $\alpha, \beta \in \supp(f)$ for which (EA) fails. Write 
$
\alpha=e_i+e_j \mbox{ and } \beta= e_k+e_\ell. 
$
Then $\{i,j\}$ and $\{k,\ell\}$ are disjoint. Let $t>0$ and 
$
\uu=e_i+te_j \mbox{ and } \vv= e_k+e_\ell.
$
Then
$$
\uu^T H \vv = h_{ik}+h_{i\ell}, \ 
\uu^TH\uu = h_{ii}+t^2h_{jj} + 2th_{ij}, \mbox{ and }
\vv^TH\vv = h_{kk}+h_{\ell \ell} + 2h_{k\ell},
$$
where $h_{ij}h_{k\ell}>0$. Hence for large $t$,
$
(\uu^T H \vv)^2 < (\uu^TH\uu) (\vv^TH\vv),
$
which contradicts (AF). 
\end{proof}

From Theorem \ref{main-L2}, Lemma \ref{quadr} and Proposition \ref{char-MC} we deduce the following characterization of Lorentzian polynomials, first proved in \cite{BH}. 

\begin{theorem}\label{main-L}
  Let $f\in \RR[t_1,\ldots, t_n]$ be a homogeneous polynomial of degree $d\geq 2$ with nonnegative coefficients. Then $f$ is Lorentzian if and only if 
\begin{itemize}
\item[(S)] the support of $f$ is M-convex, and 
\item[(H1)] for any $1\leq i_3, \ldots, i_d \leq n$, the symmetric bilinear form 
$$
(\xx, \yy) \longmapsto D_{\xx}D_{\yy} \partial_{i_3}\cdots \partial_{i_d} f
$$
has at most one positive eigenvalue.
\end{itemize}
 \end{theorem}

 The following proposition is a consequence of Propositions \ref{CL-equal} and Remark \ref{AF-remark}. 
 
 \begin{proposition}\label{redu}
 Let $f \in \RR[t_1,\ldots, t_n]$ be a homogeneous polynomial of degree $d$, and $\CCC \subset \RR^n$ an open convex cone. Then $f$ is $\CCC$-Lorentzian if and only if for all $\vv_1,\ldots, \vv_d \in \CCC$, the polynomial 
 $$
 f(s_1\vv_1+\cdots + s_d\vv_d) \in \RR[s_1,\ldots, s_d]
 $$
 is Lorentzian. 
 \end{proposition}

\rev{Denote by $\LLL_n(\CCC)$ the space of all $\CCC$-Lorentzian polynomials in $\RR[t_1,\ldots, t_n]$, and let $\LLL_n^d(\CCC) \subset \LLL_n(\CCC)$ be the space of all such polynomials of degree $d$. 
Given an open convex cone $\CCC$ in $\RR^n$ with lineality space $L_{\CCC}$, denote by $\Hmg_n^d(\CCC)$ the linear subspace of $d$-homogeneous polynomials in $\RR[t_1,\ldots,t_n]$ for which $L_{\CCC}$ is contained in the lineality space of $f$. Hence $\LLL_n^d(\CCC) \subseteq \Hmg_n^d(\CCC)$. Further, denote by $\rho_{\CCC}: \RR^n \to \RR^n / L_{\CCC}$ the standard linear projection, and let $S_{\CCC} \subset \overline{\CCC}$ be a set of unit vectors for which the projected vectors $\rho_{\CCC}(S_{\CCC})$ generate the extreme rays of $\overline{\rho_{\CCC}(\CCC)} \subset \RR^n / L_{\CCC}$.}

 The next theorem is a version of Theorem \ref{main-L} for $\CCC$-Lorentzian polynomials.

 \begin{theorem} \label{M-convex-extreme}
     \rev{Let $\CCC$ be an open convex cone, let $f \in \Hmg_n^d(\CCC)$ be of degree $d \geq 2$, and let $S_{\CCC}$ be defined as above}. Then $f$ is $\CCC$-Lorentzian if and only if
     \begin{itemize}
     \item $D_{\uu_1}D_{\uu_2}\cdots D_{\uu_d}f\geq 0$ for all $\uu_1,\ldots,\uu_d \in S_{\CCC}$, and 
         \item the set $\{ \alpha \in \NN^{2d} : \alpha_1+\cdots+\alpha_{2d}=d \mbox{ and } D_{\uu_1}^{\alpha_1}\cdots D_{\uu_{2d}}^{\alpha_{2d}}f>0\}$ is M-convex for all $\uu_1,\ldots,\uu_{2d} \in S_{\CCC}$, and 
         \item for all $\uu_3,\ldots,\uu_d \in S_{\CCC}$, the symmetric bilinear form 
            $$
            (\xx, \yy) \longmapsto D_{\xx}D_{\yy} D_{\uu_3} \cdots D_{\uu_d} f
            $$
            has at most one positive eigenvalue.
     \end{itemize}
     
 \end{theorem}

 \begin{proof}
\rev{Note that up to translation by some vector in $L_{\CCC}$, each vector of $\CCC$ is a non-negative linear combination of vectors in $S_{\CCC}$.} By Propositions \ref{redu} and \ref{comp}, $f$ is $\CCC$-Lorentzian if and only if for all vectors $\uu_1, \ldots, \uu_m \in S_{\CCC}$, the polynomial $g=f(s_1\uu_1+ \cdots+ s_m \uu_m)$ is Lorentzian. Notice that it requires at most $2d$ vectors to verify that the support of $g$ is M-convex. The proof now follows from Proposition \ref{char-MC} (using Remark \ref{closedc} for the first condition). 
 \end{proof}

\begin{corollary}\label{products}
 If $f, g \in \LLL_n(\CCC)$, then $fg \in \LLL_n(\CCC)$. 
\end{corollary}
\begin{proof}
    We will show that $f(\xx)g(\zz) \in \LLL_{2n}(\CCC \oplus \CCC)$, and the result then follows from Proposition \ref{comp} using $A: \xx \longmapsto (\xx,\xx)$. \rev{First note that $D_{\vv} (fg) = 0$ for $\vv \in L_{\CCC \oplus \CCC} = L_{\CCC} \oplus L_{\CCC}$, so that $fg \in H_{2n}(\CCC \oplus \CCC)$.} Next note that vectors that generate \rev{extreme rays of $\overline{\rho_{\CCC \oplus \CCC}(\CCC \oplus \CCC})$ are of the form $(\rho_{\CCC}(\vv),\mathbf{0})$ or $(\mathbf{0},\rho_{\CCC}(\vv))$ for some vector $\vv$ for which $\rho_{\CCC}(\vv)$ generates an extreme ray of $\overline{\rho_{\CCC}(\CCC)}$.} Thus $f(\xx)g(\zz) \in \LLL_{2n}(\CCC \oplus \CCC)$ follows from Theorem \ref{M-convex-extreme}.
\end{proof}

\begin{theorem}\label{interior}
Let $\CCC$ be an open convex cone in $\RR^n$ with lineality space $L_{\CCC}$, and \rev{let $f \in \Hmg_n^d(\CCC)$ be of degree $d\geq 2$.} Then $f$ is in the interior of $\LLL_n^d(\CCC) \subset \Hmg_n^d(\CCC)$ if and only if for all \rev{$\vv_1, \ldots, \vv_d \in \overline{S_{\CCC}}$},
\begin{itemize}
\item[(PC)] $D_{\vv_1}\cdots D_{\vv_{d}} f>0$, and 
\item[(HC)] the symmetric bilinear form 
$$
(\xx, \yy) \longmapsto D_{\xx}D_{\yy} D_{\vv_3}\cdots D_{\vv_{d}} f
$$
has exactly one positive eigenvalue, and kernel equal to $L_{\CCC}$. 
\end{itemize}

\end{theorem}
\begin{proof}
    Letting $S$ denote the set of polynomials satisfying (PC) and (HC), it is clear that $S \subseteq \LLL_n^d(\CCC)$ by Theorem \ref{M-convex-extreme}. By compactness of \rev{$\overline{S_{\CCC}}$} and continuity, (PC) and (HC) are open conditions. Hence $S$ is an open set contained in $\LLL_n^d(\CCC)$. 
    
    Finally, a standard perturbation argument shows that anything in the interior of $\LLL_n^d(\CCC)$ must be contained in $S$, and this completes the proof.
\end{proof}

The M-convexity condition of Theorem \ref{M-convex-extreme} can also be replaced by another derivative condition as follows.

\begin{theorem} \label{extreme-plus-one}
    \rev{Let $\CCC$ be an open convex cone, and let $f \in \Hmg_n^d(\CCC)$ be of degree $d \geq 2$.} Given any $\ww \in \CCC$, $f$ is $\CCC$-Lorentzian if and only if for all $\uu_1,\ldots,\uu_{d-2} \in S_{\CCC}$ and all $0 \leq k \leq d-2$, the quadratic polynomial
    \[
        D_{\uu_1} \cdots D_{\uu_k} D_{\ww}^{d-2-k} f
    \]
    is $\CCC$-Lorentzian.
\end{theorem}
\begin{proof}
    If $f$ is $\CCC$-Lorentzian, the
    \rev{fact that the space of $\CCC$-Lorentzian polynomials is closed implies}
    that $D_{\uu_1} \cdots D_{\uu_k} D_{\ww}^{d-2-k} f$ is $\CCC$-Lorentzian for all choices of $\uu_1,\ldots,\uu_{d-2}$ and $k$. For the other direction, we prove that $f$ is $\CCC$-Lorentzian by induction on $d$, the case when $d=2$ being immediate. For $d \geq 3$ fix $\vv_1, \ldots, \vv_d \in \CCC$, and let $\uu_1,\ldots,\uu_m \in S_{\CCC}$ be such that $\vv_1, \ldots, \vv_d$ and $\ww$ are contained in the relative interior of the convex hull of $\uu_1,\ldots,\uu_m$, \rev{up to translation by vectors in $L_{\CCC}$}. Consider the polynomial
    \[
        g(\sss) = f(s_1 \uu_1 + \cdots +s_m \uu_m).
    \]
  Then $\partial_i g$ is Lorentzian for all $i \in [m]$ by induction and Propositions \ref{comp} and \ref{CL-equal}. Further, there exists $\yy \in \RR_{\geq 0}^m$ such that $\sum_i y_i \uu_i = \ww$ \rev{up to translation by vectors in $L_{\CCC}$}, and thus $D_{\yy} g$ is Lorentzian as well by the same argument. Letting $\MM$ denote the support of $g$, Theorem \ref{main-L} then implies $\partial_i(\MM)$ and $\tau\MM$ are M-convex. By Lemma \ref{M-tM-connected} and Proposition \ref{alt-M}, this implies $\MM$ is M-convex. Thus by Theorem \ref{main-L} again, $g$ is Lorentzian. Since $\vv_1, \ldots, \vv_d$ are contained in the convex hull of $\uu_1,\ldots,\uu_m$ \rev{up to translation by vectors in $L_{\CCC}$}, Propositions \ref{comp} and \ref{CL-equal} then imply
    $
        f(s_1 \vv_1 + \cdots + s_d \vv_d)
    $
    is Lorentzian. Since this holds for all choices of $\vv_1, \ldots, \vv_d \in \CCC$, Proposition \ref{redu} implies $f$ is $\CCC$-Lorentzian.
\end{proof}

An immediate corollary is the following characterization of Lorentzian polynomials. 

\begin{corollary}
    Let $f \in \RR[t_1,\ldots,t_n]$ be a homogeneous polynomial of degree $d \geq 2$. Then $f$ is Lorentzian if and only if for all $i_1,\ldots,i_{d-2} \in [n]$ and all $0 \leq k \leq d-2$, the quadratic polynomial
    \[
        \partial_{i_1} \cdots \partial_{i_k} D_{\mathbf{1}}^{d-2-k} f
    \]
    is Lorentzian.
\end{corollary}

\section{Topology of  spaces of $\CCC$-Lorentzian polynomials}\label{top}
Equip the linear space $\Hmg_n^d(\CCC)$ with a Euclidean topology. Fix $\vv \in \CCC$, and let 
$$
\mathbb{P}\LLL_n^d(\CCC)= \{ f \in \LLL_n^d(\CCC) : f(\vv)=1 \}. 
$$
Since $\LLL_n^d(\CCC)$ is closed, by Remark \ref{closedc}, it follows that $\mathbb{P}\LLL_n^d(\CCC)$ is compact in $\Hmg_n^d(\CCC)$. 
We end by proving that the space $\mathbb{P}\LLL_n^d(\CCC)$ is contractible, which generalizes \cite[Theorem 2.28]{BH} to any cone $\CCC$. 

\begin{theorem}\label{clopen}
Let $\CCC$ be an open convex cone in $\RR^n$. The space $\mathbb{P}\LLL_n^d(\CCC) \subset \Hmg_n^d(\CCC)$ is the closure of its interior, and   is contractible to a polynomial in its interior. 

\end{theorem}

\begin{proof}
 We will construct a deformation retract  $F :  \mathbb{P}\LLL_n^d(\CCC)\times [0,1] \longrightarrow \mathbb{P}\LLL_n^d(\CCC)$, such that for all 
 $f \in \mathbb{P}\LLL_n^d(\CCC)$,
 \begin{enumerate}
     \item $F(f,s)$ is in the interior of $\mathbb{P}\LLL_n^d(\CCC)$  for all $0 <s \leq 1$, and 
     \item $F(f,1)=g$, for a specific polynomial $g$ in the interior of $\mathbb{P}\LLL_n^d(\CCC)$.  
     \end{enumerate}
 
 We may assume $\CCC$ is pointed by considering the quotient of $\RR^n$ and $\CCC$ by $L_{\CCC}$. Choose linearly independent vectors $\ww_1,\ldots, \ww_n$ in 
$$
\CCC^*= \{\ww \in \RR^n : \langle \ww, \uu \rangle >0 \mbox{ for all } \uu \in \overline{\CCC} \cap \mathbb{S}^{n-1} \},
$$
the interior of the dual cone of $\overline{\CCC}$. Fix $\vv \in \CCC$, and let $\ww= \ww_1+\cdots+\ww_n$. For $s>0$, let
$$
f_s(\ttt)= f(\ttt+s \langle \ttt, \ww\rangle\vv)- Cs^d f(\vv)\sum_{j=1}^n \langle \ttt, \ww_j\rangle^d, 
$$
where $0<C<1$. 
We will prove that $f_s$ lies in the interior of $\LLL_n^d(\CCC)$. 
By Proposition~\ref{comp},  $g_s(\ttt)= f(\ttt+s \langle \ttt, \ww\rangle\vv)$ is in $\LLL_n^d(\CCC)$. Moreover by \eqref{us},
$$
D_{\uu_1}\cdots D_{\uu_d}g_s \geq  s^d D_\vv^d f \prod_{i=1}^d  \langle \uu_i, \ww \rangle = d! s^d f(\vv)  \prod_{i=1}^d \langle \uu_i, \ww \rangle,
$$
whenever $\uu_i \in \overline{\CCC} \cap \mathbb{S}^{n-1}$ for all $i$. Also,
$$
d! \prod_{i=1}^d \langle \uu_i, \ww \rangle \geq d! \sum_{j=1}^n \prod_{i=1}^d \langle \uu_i, \ww_j \rangle = D_{\uu_1}\cdots D_{\uu_d} \sum_{j=1}^n \langle \ttt, \ww_j\rangle^d.
$$
From this follows $D_{\uu_1}\cdots D_{\uu_d}f_s >0$, whenever  $\uu_i \in \overline{\CCC} \cap \mathbb{S}^{n-1}$ for all $i$. Moreover 
$$
D_{\uu_3}\cdots D_{\uu_d}f_s = D_{\uu_3}\cdots D_{\uu_d}g_s - q(\ttt),
$$
where $q(\ttt)$ is positive definite. Hence the Hessian of $D_{\uu_3}\cdots D_{\uu_d}f_s$ has exactly one positive eigenvalue and $(d-1)$ negative eigenvalues. By Theorem~\ref{interior}, $f_s$ is a polynomial in the interior of $\LLL_n^d(\CCC)$ for all $s>0$. 

Now 
$$
\lim_{s \to \infty} \frac {f_s(\ttt)} {f_s(\vv)} = \frac {\langle \ttt, \ww\rangle^d- C\sum_{j=1}^n \langle \ttt, \ww_j\rangle^d}{\langle \vv, \ww\rangle^d- C\sum_{j=1}^n \langle \vv, \ww_j\rangle^d}, 
$$
is a polynomial in the interior of $\mathbb{P}\LLL_n^d(\CCC)$. Hence 
$$
F(f,s)= \frac {f_{s/(1-s)}(\ttt)} {f_{s/(1-s)}(\vv)},
$$
is a deformation retract satisfying the required properties.
\end{proof}

It was proved in \cite{Bball} that $\mathbb{P}\LLL_n^d(\RR_{>0}^n)$ is homeomorphic to a Euclidean ball. 
\begin{question}
Is $\mathbb{P}\LLL_n^d(\CCC)$ homeomorphic to a Euclidean ball? What about if $\CCC$ is a polyhedral cone? 
\end{question}

\noindent
\textbf{Acknowledgements.} \rev{The authors would like to thank Weizhe Zheng for pointing out two errors in Section~\ref{Lorsec} in a previous version of this manuscript. The authors would also like to thank the anonymous referees for their careful reading and many comments, which improved the quality of the paper.}

\end{document}